\documentclass{amsart}
\usepackage[utf8]{inputenc}
\usepackage{cite}
\usepackage{amsfonts}
\usepackage{amssymb}
\usepackage{amsthm}
\usepackage{amsmath}
\usepackage{graphicx}
\usepackage{stmaryrd}
\usepackage{csquotes}
\usepackage{enumerate}
\usepackage{aliascnt}
\usepackage{mathtools}
\usepackage{bbm}
\usepackage[unicode,bookmarksopen]{hyperref}

\newtheorem{theorem}{Theorem}[section]

\newcommand{\MakeTheoremAndCounter}[2]{\newaliascnt{#1}{theorem}
\newtheorem{#1}[#1]{#2}
\aliascntresetthe{#1}
\expandafter\providecommand\csname#1autorefname\endcsname{#2}}

\MakeTheoremAndCounter{lemma}{Lemma}
\MakeTheoremAndCounter{claim}{Claim}
\MakeTheoremAndCounter{subclaim}{Subclaim}
\MakeTheoremAndCounter{example}{Example}
\MakeTheoremAndCounter{fact}{Fact}
\MakeTheoremAndCounter{question}{Question}
\MakeTheoremAndCounter{corollary}{Corollary}
\MakeTheoremAndCounter{notation}{Notation}
\MakeTheoremAndCounter{problem}{Problem}
\MakeTheoremAndCounter{proposition}{Proposition}
\MakeTheoremAndCounter{conjecture}{Conjecture}

\newcommand{\thistheoremname}{}
\newtheorem*{genericthm*}{\thistheoremname}
\newenvironment{namedthm*}[1]
{\renewcommand{\thistheoremname}{#1}%
	\begin{genericthm*}}
	{\end{genericthm*}}

\theoremstyle{definition}
\MakeTheoremAndCounter{definition}{Definition}
\newtheorem*{remark}{Remark}

\numberwithin{equation}{section}

\newcommand{\HomeoId}{\Homeo([0,1]^d)}

\DeclareMathOperator{\inter}{int}

\DeclareMathOperator{\range}{range}
\DeclareMathOperator{\length}{length}
\DeclareMathOperator{\Lip}{Lip}
\DeclareMathOperator{\graph}{graph}

\DeclareMathOperator{\pr}{pr}
\DeclareMathOperator{\Homeo}{Homeo}
\DeclareMathOperator{\id}{id}
\DeclareMathOperator{\diam}{diam}

\newcommand{\R}{\mathbb{R}}
\newcommand{\N}{\mathbb{N}}
\newcommand{\eps}{\varepsilon}
\newcommand{\iB}{\mathcal{B}}
\newcommand{\iK}{\mathcal{K}}
\newcommand{\iI}{\mathcal{I}}
\newcommand{\iF}{\mathcal{F}}
\newcommand{\iG}{\mathcal{G}}
\newcommand{\iH}{\mathcal{H}}

\newcommand{\defeq}{\stackrel{\text{def}}{=}}

\newcommand{\titleofthepaper}{Singularity of maps of several variables and a problem of Mycielski concerning prevalent homeomorphisms}
\hypersetup{pdftitle=\titleofthepaper}
\begin{document}
\title[Singularity of functions of several variables]{\titleofthepaper}

\author{Rich\'ard Balka}
\address{Alfr\'ed R\'enyi Institute of Mathematics, Re\'altanoda u.~13--15, H-1053 Budapest, Hungary}
\email{balka.richard@renyi.hu}

\author{M\'arton Elekes}
\address{Alfr\'ed R\'enyi Institute of Mathematics, Re\'altanoda u.~13--15, H-1053 Budapest, Hungary AND E\"otv\"os Lor\'and University, Institute of Mathematics, P\'azm\'any P\'eter s. 1/c, 1117 Budapest, Hungary}
\email{elekes.marton@renyi.hu}
\urladdr{http://www.renyi.hu/$\sim$emarci}

\author{Viktor Kiss}
\address{Alfr\'ed R\'enyi Institute of Mathematics, Re\'altanoda u.~13--15, H-1053 Budapest, Hungary}
\email{kiss.viktor@renyi.hu}

\author{M\'ark Po\'or}
\address{E\"otv\"os Lor\'and University, Institute of Mathematics, P\'azm\'any P\'eter s.~1/c, 1117 Budapest, Hungary, and Einstein Institute of Mathematics Edmond J. Safra Campus, Givat Ram The Hebrew University of Jerusalem, 91904 Jerusalem, Israel}
\email{sokmark@gmail.com}

\thanks{The authors were supported by the National Research, Development and Innovation Office -- NKFIH, grants no.~113047, 129211 and 124749. The first author was supported by the MTA Premium Postdoctoral Research Program. The third author was supported by the National Research, Development and Innovation Office -- NKFIH, grant no.~128273.\\
\includegraphics[height=1cm]{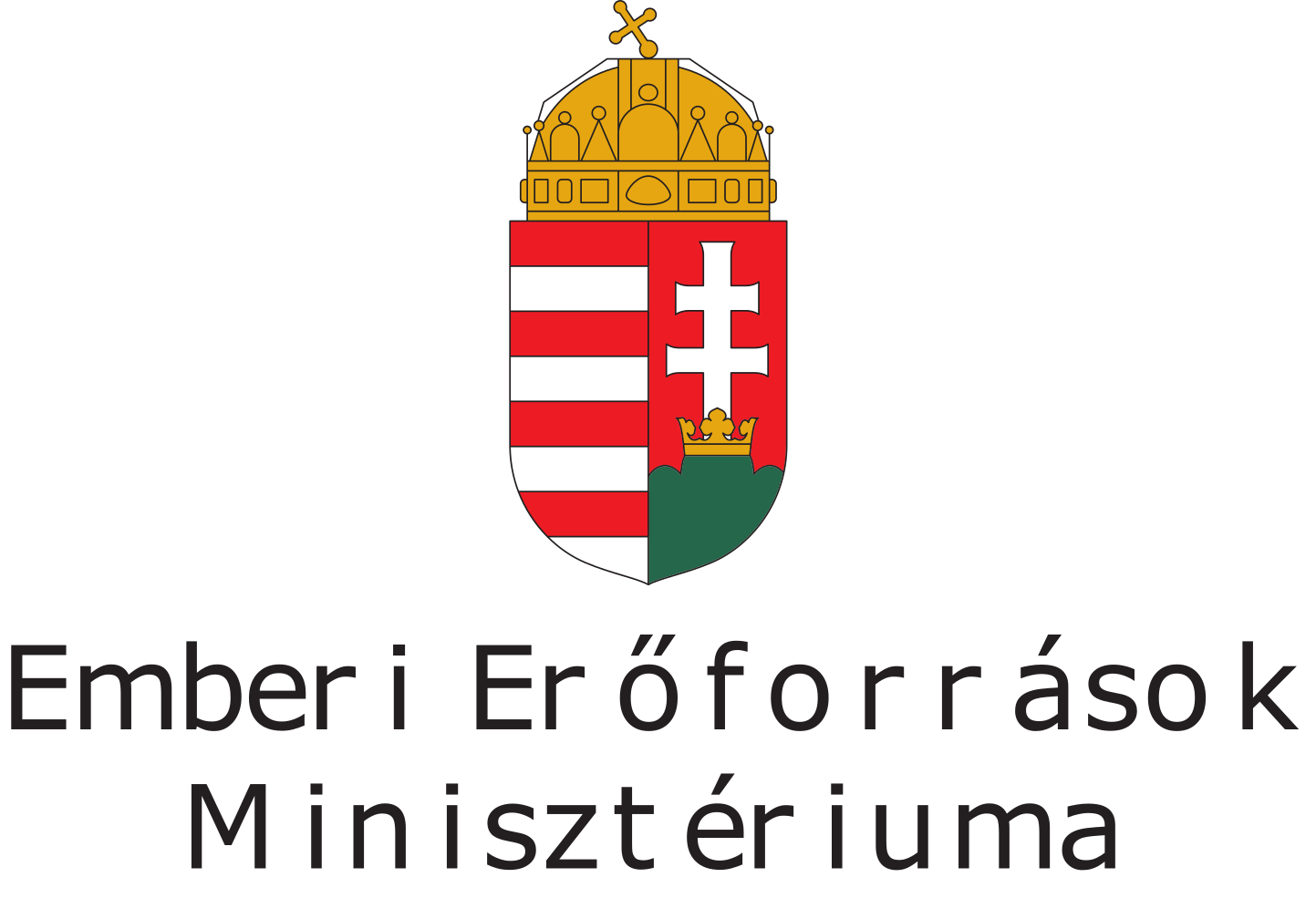} \raisebox{0.5cm}{\parbox[c]{11cm}{The fourth author was supported through the New National Excellence Program of the Ministry of Human Capacities.}}}



\begin{abstract}
S.~Banach pointed out that the graph of the generic (in the sense of Baire category) element of $\Homeo([0,1])$ has length $2$. J.~Mycielski asked if the measure theoretic dual holds, i.e., if the graph of all but Haar null many (in the sense of Christensen) elements of $\Homeo([0,1])$ have length $2$. We answer this question in the affirmative. 

We call $f \in \HomeoId$ singular if it takes a suitable set of full measure to a nullset, and strongly singular if it is almost everywhere differentiable with singular derivative matrix. 
Since the graph of $f \in \Homeo([0,1])$ has length $2$ iff $f$ is singular iff $f$ is strongly singular, the following results are the higher dimensional analogues of Banach's observation and our solution to Mycielski's problem.

We show that for $d \ge 2$ the graph of the generic element of $\Homeo([0,1]^d)$ has infinite $d$-dimensional Hausdorff measure, contrasting the above result of Banach. The measure theoretic dual remains open, but we show that the set of elements of $\Homeo([0,1]^d)$ with infinite $d$-dimensional Hausdorff measure is not Haar null. We show that for $d \ge 2$ the generic element of $\Homeo([0,1]^d)$ is singular but not strongly singular. We also show that for $d \ge 2$ almost every element of $\Homeo([0,1]^d)$ is singular, but the set of strongly singular elements form a so called Haar ambivalent set (neither Haar null, nor co-Haar null).

Finally, in order to clarify the situation, we investigate the various possible definitions of singularity for maps of several variables, and explore the connections between them.
\end{abstract}

\maketitle

\tableofcontents

\section{Introduction}

Let $\HomeoId$ denote the group of homeomorphisms of the $d$-dimensional cube $[0,1]^d$, and let us equip this group with the maximum metric. It is well known that this is a Polish group \cite[9.B.8]{Ke}, i.e., a completely metrizable separable topological group, hence we can apply Baire category arguments. (Note that the maximum metric itself is not complete, but instead the topology it generates is completely metrizable.)

According to J.~Mycielski~\cite{My}, the following interesting but easy fact is due to S.~Banach.

\begin{fact}[S.~Banach]
\label{f:Banach}
The graph of the generic (in the sense of Baire category) element of $\Homeo([0,1])$ is of length $2$. 
\end{fact}

Length here can be considered as the arclength of the planar curve $x \mapsto (x, f(x))$, but it is well known that for injective continuous curves arclength of a curve is the same as the 1-dimensional Hausdorff measure of the range of the curve, i.e., in this case $\mathcal{H}^1(\graph(f))$. (See the next section for the definitions.)

Mycielski \cite{My} raised the question if the measure theoretic dual of Fact~\ref{f:Banach} holds. However, since $\Homeo([0,1])$ carries no natural invariant measure, first we need to clarify what he meant by almost every homeomorphism. The following notion was  introduced by J.~P.~R.~Christensen \cite{Ch}.

\begin{definition}[J.~P.~R.~Christensen] 
A subset $X$ of a Polish group $G$ is \emph{Haar null} if there exists a Borel probability measure $\mu$ and a Borel set $B$ containing $X$ such that $\mu(gBh)=0$ for every $g,h \in G$. The complement of a Haar null set is called \emph{prevalent}.
\end{definition}

Christensen proved that Haar null sets form a proper $\sigma$-ideal which coincides with the family of sets of Haar measure zero if $G$ is locally compact. This notion turned out to be very useful in various branches of mathematics, see e.g.~the survey paper \cite{EN}. 

Now we are able to formulate Mycielski's question \cite{My}.

\begin{question}[J.~Mycielski]
\label{q:Mycielski} What is the length of the graph of the prevalent element of $\Homeo([0,1])$?
\end{question}

One of the goals of the present paper is to answer this question. 

\begin{namedthm*}{Corollary~\ref{c:Mycielski}}
The graph of the prevalent $f\in \Homeo([0, 1])$ is of length $2$.
\end{namedthm*}

It turns out that this result is closely related to the following notions. 

\begin{definition} We say that $f\in \HomeoId$ is \emph{singular} if there exists a Borel set $F \subseteq [0,1]^d$ such that $\lambda^d(F)=1$ and $\lambda^d(f(F)) = 0$, and  \emph{strongly singular} if $f$ is differentiable at $x$ and $\det f'(x) = 0$ for almost every $x \in [0,1]^d$. 
\end{definition} 

In Theorem \ref{t:length} below we will prove that the graph of $f \in \Homeo([0,1])$ has length $2$ iff $f$ is singular iff $f$ is strongly singular. Therefore the following problems are all natural generalizations of Banach's observation and Mycielski's problem.

\begin{question}
\label{q:Hd}
Let $d \ge 2$ be an integer. What is the $d$-dimensional Hausdorff measure of the graph of the generic/prevalent element of $\Homeo([0,1]^d)$?
\end{question}

\begin{question}
\label{q:sing}
Let $d \ge 2$ be an integer. Is the generic/prevalent element of $\Homeo([0,1]^d)$ singular? 
\end{question}

\begin{question}
\label{q:strong_sing}
Let $d \ge 2$ be an integer. Is the generic/prevalent element of $\Homeo([0,1]^d)$ strongly singular?
\end{question}

In Theorem~\ref{t:inf} we will answer Question \ref{q:Hd} in the generic case by showing that for $d \ge 2$ the graph of the generic element of $\Homeo([0,1]^d)$ has infinite $d$-dimensional Hausdorff measure, contrasting the above result of Banach. The prevalent case remains open, but we show in Theorem~\ref{t:ccinf} the partial result that the set of elements of $\Homeo([0,1]^d)$ with infinite $d$-dimensional Hausdorff measure is not Haar null. We also answer the remaining two questions in Sections \ref{s:gen_sing} and \ref{s:prev_sing}: We show that for $d \ge 2$ the generic element of $\Homeo([0,1]^d)$ is singular but not strongly singular, and also that for $d \ge 2$ the prevalent element of $\Homeo([0,1]^d)$ is singular, but the set of strongly singular elements form a so called Haar ambivalent set (neither Haar null, nor prevalent).

Moreover, in Section~\ref{s:sing}, in order to clarify the situation, we investigate the various possible definitions of singularity for maps of several variables, and explore the connections between them.

Finally, in the Appendix we prove an approximation result communicated to us by J.~Luukkainen stating that the set of somewhere smooth homeomorphisms is dense in $\Homeo([0,1]^d)$. These types of problems are really delicate, dimension $4$ is particularly difficult, since for example there exists a homeomorphism $f$ of $(0,1)^4$ into $\R^4$ that cannot be uniformly approximated even by bi-Lipschitz homeomorphisms, let alone smooth ones \cite[page~183]{DS}. Moreover, J.~Luukkainen \cite{Lu} pointed out to us that we can even find such an $f\in \Homeo([0,1]^4)$. 

We note here that several problems we consider in this paper have well-known analogues for continuous maps instead of homeomorphisms. For example it is classical that the generic $f\in C([0, 1]^d, \R^d)$ is nowhere differentiable, (the case of $d=1$ dates back to Banach \cite{B}, for the general case see e.g.~Claim~\ref{c:dif2}), hence not strongly singular, but singular (see Theorem~\ref{t:versions of singularity}). The prevalent $f\in C([0,1],\R)$ is nowhere differentiable \cite{H}, hence not strongly singular, and also not singular, since quite the opposite holds, namely, the occupation measure $\lambda\circ f^{-1}$ is absolutely continuous \cite[Theorem~4.2]{BDE}. (It will be shown in Theorem~\ref{t:versions of singularity} that $f$ is singular iff the occupation measure is a singular measure.)
But it should be noted that the case of homeomorphisms is very different and much more difficult. First, as we have just mentioned above, the problem of approximating homeomorphism with more regular homeomorphisms is very subtle. Second, $\HomeoId$ is not commutative,
and the theory of Haar null sets is notoriously complex for non-abelian groups (the group structure on $C([0, 1]^d, \R^d)$ is pointwise addition, hence commutative).

As our final remark, we mention that there is a completely different approach to define random homeomorphisms, see Graf, Mauldin, and Williams~\cite{GMW}.

\section{Preliminaries and notations}

Let $G$ be a Polish group. A set $A\subseteq G$ is called \emph{compact catcher} if for every compact set $K\subseteq G$ there exist $g,h\in G$ such that $gKh\subseteq A$. The following fact is well known, see e.g.~\cite[Lemma~6.6.1]{EN}. 

\begin{fact} If $A$ is compact catcher then it is not Haar null.
\end{fact} 

We use $\lambda^d$ to denote the $d$-dimensional Lebesgue measure on $\mathbb{R}^d$, and we simply write $\lambda$ instead of $\lambda^1$. Let $|\cdot|$ stand for the Euclidean norm on $\R^d$. Let $B(x,r)$ and $U(x,r)$ denote the closed and open balls of radius $r$ around $x$, respectively. We say that $K\subseteq \R^d$ is a \emph{Cantor set} if it is perfect and totally disconnected. For a set $A$ let $\diam A$, $\partial A$, and $\inter A$ stand for the diameter, the boundary, and the interior of $A$, respectively. For a sequence of sets $\{A_n\}_{n\geq 1}$ let 
\begin{equation*}
\liminf_n A_n=\bigcup_{n=1}^{\infty} \bigcap_{k=n}^{\infty} A_k     \quad \text{and} \quad \limsup_n A_n=\bigcap_{n=1}^{\infty}  \bigcup_{k=n}^{\infty} A_k.
\end{equation*}
Let $\HomeoId$ denote the group of homeomorphisms of the $d$-dimensional cube $[0,1]^d$ (under composition, as the group operation), and let us equip this group with the maximum metric. It is well known that this is a Polish group \cite[9.B.8]{Ke}, i.e., a completely metrizable separable topological group. Note that the maximum metric itself is not complete, but instead the topology it generates can also be generated by the complete metric $\varrho(f, g)=\|f -g\| + \|f^{-1} - g^{-1}\|$, where $\|\cdot \|$ denotes the maximum norm. We will also often work with the Polish space of continuous functions $f\colon [0,1]^d\to \R^d$ endowed with the maximum metric as well. Let $B(f,r)$ denote the closed ball of radius $r$ around $f$ with respect to the maximum metric.

We use $\mathcal{B}(X)$ to denote the Borel sets of a topological space $X$ and $\mathbf{P}(X)$ to denote the space of the Borel probability measures on a separable metrizable space $X$. The classical result \cite[Theorem~17.19]{Ke} states that $\mathbf{P}(X)$ is also a separable metrizable space when equipped with the weak topology.

For $A\subseteq \mathbb{R}^{d}$ and $s \ge 0$ we define
\emph{$s$-dimensional Hausdorff measure} as
\begin{align*}
\mathcal{H}^{s}(A)&=\lim_{\delta\downarrow  0}\mathcal{H}^{s}_{\delta}(A),
\text{ where}\\
\mathcal{H}^{s}_{\delta}(A)&=\inf \left\{ \sum_{i=1}^\infty (\diam
A_{i})^{s}: A \subseteq \bigcup_{i=1}^{\infty} A_{i},~
\forall i \diam A_i \le \delta \right\}.
\end{align*}

The \emph{length of a curve} $ \gamma \colon [0,1]\to \mathbb{R}^{d}$ is defined as
\begin{equation*} \length (\gamma)=\sup \left\{\sum_{i=1}^{n} |\gamma(x_{i})-\gamma(x_{i-1})|:
n\in \mathbb{N}^{+},~ 0=x_0<\dots<x_n=1 \right\}.
\end{equation*}

The 
\emph{length of the graph} of a function $f \colon [0,1] \to [0, 1]$ is defined as 
\begin{equation*}
\length(\graph(f)) = \length(\gamma),
\end{equation*}
where $\gamma \colon [0, 1] \to [0, 1]^2$ is defined as $\gamma(x) = (x, f(x))$. Since $\gamma$ is one-to-one, the following well-known fact follows e.g.~from \cite[Theorem~2.6.2]{BBI}. 
\begin{fact} \label{f:wk} 
For each continuous function $f\colon [0,1]\to [0,1]$ we have 
\begin{equation*}
\length(\graph(f))=\iH^1(\graph(f)).    
\end{equation*}
\end{fact} 
    
Let $X$ be a completely metrizable topological space. A subset of $X$ is \emph{somewhere dense} if it is dense in a non-empty open set, otherwise it is called \emph{nowhere dense}. We say that $M \subseteq X$ is
\emph{meager} if it is a countable union of nowhere dense sets, and a set is called \emph{co-meager} if its complement is meager. It is not difficult to show that a set is co-meager iff it contains a dense
$G_\delta$ set. We say that the generic element $x \in X$ has property $\mathcal{P}$ if $\{x \in X : x \textrm{ has property }
\mathcal{P} \}$ is co-meager. See e.g.~\cite{Ke} for more on these concepts.

Let $H\subseteq \R^d$ and $f\colon H\to \R^d$. Assume that $x_0\in H$ and $x_0$ is a limit point of $H$. We call $f$ \emph{differentiable at $x_0$}  if there is a $d\times d$ matrix $L$ such that
\begin{equation} \label{e:diff} 
\lim_{x\to x_0, \, x\in H} \frac{|f(x)-f(x_0)-L(x-x_0)|}{|x-x_0|}=0. 
\end{equation}
The above matrix $L$ is not necessarily unique. We say that $f$ is \emph{not differentiable at $x_0$} if there exists no $L$ satisfying \eqref{e:diff}. Actually, we will always prove in this case the stronger property that
\begin{equation*}
\limsup_{x\to x_0, \, x\in H} \frac{|f(x)-f(x_0)|}{|x-x_0|}=\infty.
\end{equation*} 
Note that if $f$ is differentiable at $x_0$ and either $H=[0,1]^d$ or $x_0$ is a Lebesgue density point of $H$ then there is a unique $L$ satisfying \eqref{e:diff}, in which case we define $f'(x_0)=L$.

\section{Versions of singularity}
\label{s:sing}

In this section we explore the connections between the various possible definitions of singularity of maps. The special cases when the map is one-to-one, or differentiable almost everywhere, or $d = 1$ are also examined. 

\begin{definition} \label{d:sing} Let $d \ge 1$, $H \subseteq \R^d$ Borel, and $f \colon H \to \R^d$ be Borel measurable. We say that $f$ is \emph{strongly singular} if $f$ is differentiable at $x$ and $\det f'(x) = 0$ for almost every $x \in H$, and $f$ is \emph{singular} if there exists a Borel set $F \subseteq H$ such that $\lambda^d(H \setminus F)=0$ and $\lambda^d(f(F)) = 0$. \end{definition} 

 \begin{theorem}
\label{t:versions of singularity}
		Let $d \ge 1$, $H \subseteq \R^d$ be Borel, and $f \colon H \to \R^d$ be Borel measurable. Consider the following properties of $f$:
		\begin{enumerate}
\item\label{f:derivative}  $f$ is strongly singular,
\item\label{f:limsup image} $\lim_{r \downarrow 0} \frac{\lambda^d(f(B(x, r)\cap H))}{\lambda^d(B(x, r))} = 0$ for almost every $x \in H$,
\item\label{f:limsup preimage} $\lim_{r \downarrow  0} \frac{\lambda^d(f^{-1}(B(y, r)))}{\lambda^d(B(y, r))} = 0$ for almost every $y \in \R^d$,
\item\label{f:pushforward is singular to lebesgue} the pushforward measure $\lambda^d \circ f^{-1}$ is singular with respect to $\lambda^d$, 
\item\label{f:full to null} $f$ is singular. 
\end{enumerate}
Then \eqref{f:derivative} $\Rightarrow$ \eqref{f:limsup image} $\Rightarrow$ \eqref{f:limsup preimage} $\Leftrightarrow$ \eqref{f:pushforward is singular to lebesgue} $\Leftrightarrow$ \eqref{f:full to null}. If $f$ is one-to-one then \eqref{f:limsup image} $\Leftrightarrow$ \eqref{f:limsup preimage}. If $f$ is differentiable at almost every $x \in H$ then \eqref{f:full to null}
$\Rightarrow$ \eqref{f:derivative}, therefore
\eqref{f:derivative} $\Leftrightarrow$ \eqref{f:limsup image} $\Leftrightarrow$ \eqref{f:limsup preimage} $\Leftrightarrow$ \eqref{f:pushforward is singular to lebesgue} $\Leftrightarrow$ \eqref{f:full to null}. 
\noindent
For each $d \ge 1$ the generic continuous map $f \colon [0, 1]^d \to \R^d$ witnesses \eqref{f:limsup preimage} $\not \Rightarrow$ \eqref{f:limsup image}. For each $d\ge 1$ and Cantor set $K \subseteq \R^d$ with $\lambda^d(K)>0$ the generic continuous map $f\colon K\to \R^d$ is one-to-one and witnesses \eqref{f:limsup image} $\not \Rightarrow$ \eqref{f:derivative}.  
\end{theorem}
	
\begin{proof} We first show the directions \eqref{f:derivative} $\Rightarrow$ \eqref{f:limsup image}, \eqref{f:limsup image} $\Rightarrow$ \eqref{f:full to null}, \eqref{f:limsup preimage} $\Leftrightarrow$ \eqref{f:pushforward is singular to lebesgue}, \eqref{f:pushforward is singular to lebesgue} $\Leftrightarrow$ \eqref{f:full to null}, \eqref{f:limsup image} $\Leftrightarrow$ \eqref{f:limsup preimage} if $f$ is one-to-one, \eqref{f:full to null} $\Rightarrow$ \eqref{f:derivative} if $f$ is differentiable at almost every $x \in H$, and then construct the counterexamples. 
		
\eqref{f:derivative} $\Rightarrow$ \eqref{f:limsup image}:  
Assume that $f'$ exists at $x\in H$ and $\det f'(x) = 0$ and let
\begin{equation*}
R = \sup\{ |f'(x)y|: |y|\leq 1\}.  
\end{equation*} 
For $\varepsilon > 0$ fixed, let $r_0 > 0$ be small enough so that 
\begin{equation*}
 |f(y) - f(x) - f'(x) (y - x)| \le \varepsilon |y - x|   
\end{equation*} for all $y \in B(x, r_0) \cap H$. Since $\det f'(x) = 0$, the set $\{f'(x)y: y\in \R^d\}$ is contained in a $(d-1)$-dimensional hyperplane. Hence $f(B(x, r)\cap H)$ is contained in the $\varepsilon r$-neighborhood of a $(d-1)$-dimensional ball of radius $Rr$ for any $r < r_0$. Using the formula for the volume of the $d$-dimensional ball, $\lambda^d(B(\mathbf{0}, r)) = C_dr^d$, it follows that 
\begin{equation*}  
\frac{\lambda^d(f(B(x, r)\cap H))}{\lambda^d(B(x, r))} \le \frac{C_d ((R + \varepsilon) r)^{d - 1} 2\varepsilon r}{C_d r^{d}} \le c_x \varepsilon
\end{equation*}
for all $r < r_0$, where $c_x$ does not depend on $r$, $r_0$ or $\varepsilon$. Therefore the $\limsup$ in question is indeed $0$.
		
\eqref{f:limsup image} $\Rightarrow$ \eqref{f:full to null}: Let 
\begin{equation*}
F = \left\{x \in H : \limsup_{r \downarrow 0} \frac{\lambda^d(f(B(x, r)\cap H))}{\lambda^d(B(x, r))} = 0\right\}.   
\end{equation*}  
Using \eqref{f:limsup image}, $\lambda^d(H \setminus F) = 0$. To show that $\lambda^d(f(F)) = 0$, we use the $5r$-covering Theorem \cite[Theorem~2.1]{Ma}. By partitioning $F$ into countably many pieces, without loss of generality, it is enough to show that $\lambda^d(f(F')) = 0$, where $F' = F \cap [0, 1]^d$.

For a fixed $\varepsilon > 0$, let 
\begin{equation*} 
\mathcal{B} = \left\{B(x, r) : x \in F'\, , r \le 1, \text{ and } \frac{\lambda^d\left(f\left(B(x, 5r) \cap H\right)\right)}{\lambda^d(B(x, 5r))} \le \varepsilon\right\}.
\end{equation*}
Clearly, $F'\subseteq \bigcup \mathcal{B}$. The $5r$-covering theorem implies that there is a countable family $\mathcal{B}' \subseteq \mathcal{B}$ consisting of pairwise disjoint balls such that $5\mathcal{B}' \defeq \bigcup \{5B : B \in \mathcal{B}'\}$ covers $\bigcup \mathcal{B}$, where $5B = B(x, 5r)$ for a ball $B = B(x, r)$. 
Then
\begin{equation*}
\lambda^d(f(F')) \le \lambda^d\left(f\left(5\mathcal{B}'\right)\right) \le \varepsilon 5^d\lambda^d\left(\bigcup \mathcal{B}'\right) \le \varepsilon 5^d 3^d,
\end{equation*}
where we used the disjointness of the balls in $\mathcal{B}'$ and the fact that $\bigcup \mathcal{B}' \subseteq [-1, 2]^d$. It follows that $\lambda^d(F') = 0$, finishing the proof. 
		
\eqref{f:limsup preimage} $\Leftrightarrow$ \eqref{f:pushforward is singular to lebesgue}: The equivalence of a measure, in this case $\lambda^d \circ f^{-1}$ being singular with respect to $\lambda^d$ and $\limsup_{r \downarrow 0} \frac{\lambda^d(f^{-1}(B(y, r))}{\lambda^d(B(y, r))} = 0$ almost everywhere is a well-known fact in the theory of differentiation of measures. In particular, it follows from \cite[Theorem~7.14]{Rudin}. 
		
\eqref{f:pushforward is singular to lebesgue} $\Leftrightarrow$ \eqref{f:full to null}: If $\lambda^d \circ f^{-1}$ is singular with respect to $\lambda^d$ then there is a Borel set $N \subseteq \R^d$ with $\lambda^d(N) = 0$ and 
\begin{equation*}
 \lambda^d(H \setminus f^{-1}(N)) = \lambda^d(f^{-1}(\R^d) \setminus f^{-1}(N)) = \lambda^d(f^{-1}(\R^d\setminus N)) = 0. 
\end{equation*} 
Hence \eqref{f:full to null} is satisfied with $F = f^{-1}(N)$. Conversely, if $F$ is a Borel set witnessing \eqref{f:full to null}, then $f(F)$ is a Lebesgue measurable set with $\lambda^d(f(F)) = 0$ and 
\begin{equation*}
\lambda^d(f^{-1}(\R^d \setminus f(F))) = \lambda^d(H \setminus f^{-1}(f(F))) \le \lambda^d(H \setminus F) = 0,    
\end{equation*}
showing that $\lambda^d \circ f^{-1}$ is singular with respect to $\lambda^d$.
		
\eqref{f:limsup image} $\Leftrightarrow$ \eqref{f:limsup preimage} if $f$ is one-to-one: We first note that since $f$ is one-to-one, using \cite[Corollary~15.2]{Ke}, $f(B)$ is Borel for every Borel set $B \subseteq H$, and $f^{-1} \colon f(H) \to \R^d$ is a Borel measurable function. We need to prove that  \eqref{f:limsup preimage} $\Rightarrow$ \eqref{f:limsup image}. Indeed, \eqref{f:limsup preimage} implies \eqref{f:full to null}, and using that $f(F)$ is Borel, \eqref{f:full to null} holds for $f^{-1}$ as well. Then \eqref{f:full to null} yields \eqref{f:limsup preimage} for $f^{-1}$, which means that \eqref{f:limsup image} holds for $f$. 
		 
 \eqref{f:full to null}
$\Rightarrow$ \eqref{f:derivative} if $f$ is differentiable at almost every $x \in H$: We note first that the sets defined throughout the proof could be proved to be Borel sets using standard methods. However, as the proof works without using that the sets in question are Borel, the reader is encouraged to think of $\lambda^d$ as an outer measure. 
		 
Using \eqref{f:full to null}, there is a Borel set $F\subseteq H$ such that $\lambda^d(H\setminus F)=0$ and $\lambda^d(f(F))=0$. Assume to the contrary that 
\begin{equation*}
 A=\{x\in F: x \text{ is a Lebesgue density point of $H$}, \exists f'(x) \text{ and } \det f'(x)\neq 0 \}    
\end{equation*} satisfies $\lambda^d (A)>0$. Let $Q = \{ q_1, q_2, \dots\}$ be a countable dense set in $A$ and for all $n,k\in \N^+$ let 
\begin{equation*}
 \textstyle A_{n, k} = \left\{x \in A \cap B(q_k, \frac{1}{n}) : |f(x) - f(y)| \ge \frac{1}{n} |x - y| \text{ for all } y \in B(q_k, \frac{1}{n}) \;\right\}.   
\end{equation*} Using that for each $x \in A$ the derivative of $f$ at $x$ exists with non-zero determinant, it is easy to check  that $\bigcup_{n=1}^{\infty} \bigcup_{k=1}^{\infty} A_{n, k} = A$. It follows that $\lambda^d(A_{n, k}) > 0$ for some $n$ and $k$. Let us fix such an $n$ and $k$. Using the definition of $A_{n, k}$, one easily checks that $f$ is one-to-one on $A_{n, k}$ and that $f^{-1}|_{f(A_{n,k})}$ is Lipschitz with Lipschitz constant at most $n$. Using also that $A_{n,k} \subseteq F$, we obtain that 
\begin{equation*}
 \lambda^d(A_{n, k}) = \lambda^d(f^{-1}(f(A_{n, k}))) \le n\lambda^d(f(A_{n, k})) \le n \lambda^d(f(F)) = 0,   
\end{equation*} a contradiction.

\eqref{f:limsup preimage} $\not \Rightarrow$ \eqref{f:limsup image}: To show that the generic continuous function satisfies \eqref{f:limsup preimage}, it is enough to show that it satisfies \eqref{f:full to null}. It is probably widely known, but we include its short proof for the sake of completeness.

\begin{claim} \label{c:generic continuous function nowhere dense set}
There is an $F_{\sigma}$ set $F\subseteq [0,1]^d$ with $\lambda^d(F)=1$ such that the generic continuous map $f \in C([0, 1]^d, \R^d)$ satisfies $\lambda^d(f(F))=0$.
\end{claim}
\begin{proof} Take a union of Cantor sets $F=\bigcup_{i=1}^{\infty} C_i\subseteq [0, 1]^d$ such that $\lambda^d(F)=1$. Fix $i,n\in \N^+$ arbitrarily. It is enough to prove that 
\begin{equation*}
\iF_{i,n}=\{f\in C([0,1]^d,\R^d): \lambda^d(f(C_i))<1/n\}
\end{equation*}	
is dense open, since then $\bigcap_{i=1}^{\infty}\bigcap_{n=1}^{\infty} \iF_{i,n}$ gives our desired co-meager set. The set $\iF_{i,n}$ is clearly open by the outer regularity of Lebesgue measure, so it suffices to check that it is dense. Let $g \in C([0, 1]^d,\R^d)$ and $\eps > 0$ be given, we need to find $f\in \iF_{i,n}\cap B(g,\eps)$. By the uniform continuity of $g$ there exists $\delta>0$ such that $|g(x) - g(y)| < \eps$ whenever $|x - y| \le \delta$. As $C_i$ is totally disconnected, there is a finite family of pairwise disjoint open sets $\{U_1, \dots, U_k\}$ such that $\diam(U_j)<\delta$ for all $1\leq j\leq k$ and $C_i\subseteq \bigcup_{j=1}^k U_j$. We can define $f\in B(g,\eps)$ such that $f|_{C_i\cap U_j}$ is constant for all $1\leq j\leq k$. Indeed, it is enough that $f|_{C_i}\in B(g|_{C_i},\eps)$, then Tietze Extension Theorem guarantees the existence of a continuous extension $f\in B(g,\eps)$. Clearly $f\in \iF_{i,n}$, and the proof is complete.  
\end{proof}
		
Now we prove that the generic continuous map $f\colon [0, 1]^d \to \R^d$ does not satisfy \eqref{f:limsup image}. Moreover, we prove the following more general claim. First we need a known fact.

\begin{fact} \label{f:1} Assume that $0<\alpha<\beta<\infty$ and  $B=B(\mathbf{0},\beta)$ is a closed ball in $\R^d$. Let $f\colon B\to \R^d$ be a continuous map such that $f\in B(\id, \alpha)$. Then $B(\mathbf{0},\beta-\alpha)\subseteq f(B)$.  
\end{fact}
		
\begin{proof} Let $y_0\in B(\mathbf{0},\beta-\alpha)$ be arbitrary, it is enough to show $y_0\in f(B)$. Set $g(x)=x-f(x)+y_0$, then $g$ maps $B$ into itself. By Brouwer Fixed Point Theorem~\cite[Proposition~4.4]{Fu} there is a point $x_0\in B$ such that $g(x_0)=x_0$, so $f(x_0)=y_0$. \end{proof}

\begin{claim} \label{c:lims}
For the generic continuous $f\in C([0,1]^d,\R^d)$ for any $x \in [0,1]^d$,
\begin{equation*}
\limsup_{r\downarrow 0} \frac{\lambda^d(f(B(x, r)\cap [0,1]^d))}{\lambda^d(B(x, r))} = \infty.
\end{equation*}
\end{claim}
		
	\begin{proof} For the sake of simplicity we denote $f(B(x, r)\cap [0,1]^d))$ by $f(B(x,r))$.  Fix $n\in \N^+$, it is enough to show that 
	\begin{equation*} 
	\begin{split}
	\iF_n=\{f\in C([0,1]^d, \R^d): \forall x\in [0,1]^d ~ \exists  y\in [0,1]^d ~ \exists  r<1/n \\
	\text{ such that } B(y,nr)\subseteq f(B(x,r))\}
	\end{split}
	\end{equation*}
	contains a dense open set, since then $\bigcap_{n=1}^{\infty} \iF_n$ will give our desired co-meager set. Assume that $g\in C([0,1]^d, \R^d)$ and $\eps>0$ are given, we will find $f\in B(g,2\eps)$  such that $B(f,\theta)\subseteq \iF_n$ for some $\theta>0$. By the uniform continuity of $g$ we can choose a positive $\delta<\eps/(6n)$ such that $|g(x)-g(y)|\leq \eps$ whenever $|x-y|\leq \delta$. Let $\{x_1,\dots,x_k\}$ be points in $[0,1]^d$ such that $\{B(x_i,\delta)\}_{1\leq i\leq k}$ are pairwise disjoint and $\bigcup_{i=1}^k B(x_i,2\delta)=[0,1]^d$. We show that there is a continuous map $f\in B(g,2\eps)$ such that for all $1\leq i\leq k$ and $x\in B(x_i,\delta)$ we have $f(x)=g(x_i)+6n(x-x_i)$. Let $C=\bigcup_{i=1}^k B(x_i,\delta)$, by Tietze Extension Theorem it is enough to prove that $f|_{C}\in B(g|_C,2\eps)$. Indeed, for all $i$ and $x\in B(x_i,\delta)$ we have 
	\begin{equation*}
		|f(x)-g(x)|\leq |g(x)-g(x_i)|+6n|x-x_i|\leq \eps+6n\delta<2\eps.
	\end{equation*} 
	Let $\theta=3n\delta$. Finally, we need to show that $B(f,\theta)\subseteq \iF_n$. Let $r=3\delta$ and fix arbitrary $x\in [0,1]^d$ and $h\in B(f,\theta)$. Choose $i\in \{1,\dots,k\}$ such that $x\in B(x_i,2\delta)$. Then $B(x_i,\delta)\subseteq B(x,3\delta)=B(x,r)$. Composing $f$ with a homothety of ratio $\frac{1}{6n}$ and applying Fact~\ref{f:1} to it with $\beta=\delta$ and $\alpha=\frac{\delta}{2}$ yields $B(g(x_i),3n\delta) \subseteq h(B(x_i,\delta))$. Thus $y=g(x_i)$ satisfies
		\begin{equation*} 
		B(y,nr)=B(g(x_i),3n\delta) \subseteq h(B(x_i,\delta)) \subseteq h(B(x,r)).
		\end{equation*} 
		Therefore $h\in \iF_n$, and the proof is complete.
		\end{proof} 
		
	\eqref{f:limsup image} $\not \Rightarrow$ \eqref{f:derivative}:  Let $K\subseteq \R^d$ be a Cantor set with $\lambda^d(K)>0$. Consider the generic continuous map $f\in C(K,\R^d)$. Then we have $\lambda^d(f(K))=0$, for the well-known argument see the proof of Claim~\ref{c:generic continuous function nowhere dense set}, so $f$ automatically satisfies \eqref{f:limsup image}. As $K$ is totally disconnected, it is widely known that the generic $f\in C(K,\R^d)$ is one-to-one, see e.g.~\cite[Lemma~2.6]{BBE}. Hence $f$ is a homeomorphism from $K$ to $f(K)$. Therefore, it is enough to prove that the generic $f\in C(K,\R^d)$ is nowhere differentiable. We prove the following more general statement. 
		
\begin{claim} \label{c:dif2} Let $(K,\rho)$ be a compact, perfect metric space. Then for the generic continuous map $f\in C(K,\R^d)$ for all $x\in K$ we have 
\begin{equation} \label{e:rho}
\limsup_{y\to x} \frac{|f(x)-f(y)|}{\rho(x,y)}=\infty.
\end{equation}
\end{claim}
\begin{proof}
For all $k,n\in \N^+$ let 
\begin{equation*}
	 \iF_{k,n}=\{f: \exists m\geq n \text{ such that }  \diam f(B(x,1/m))>k/m \text{ for all } x\in K\}.   
\end{equation*}
Clearly the functions in $\bigcap_{k=1}^{\infty} \bigcap_{n=1}^{\infty} \iF_{k,n}$ satisfy \eqref{e:rho}, so it is enough to prove that the $\iF_{k,n}$ are dense open sets. 
		
Fix $k,n\in \N^+$. First we prove that $\iF_{k,n}$ is open. Assume that $f_i\in \iF_{k,n}^c$ and $f_i\to f$ uniformly as $i\to \infty$, we need to prove that $f\in \iF_{k,n}^c$. Fix $m\geq n$ arbitrarily. Then there exist $x_{i,m} \in K$ for all $i\in \N^+$ such that $\diam f_i(B(x_{i,m},1/m)) \leq k/m$. As $K$ is compact, by choosing a subsequence we may assume that $x_{i,m}\to x_m$ as $i\to \infty$. Then clearly $\diam f(B(x_m,1/m))\leq k/m$. As this holds for all $m\geq n$, we obtain that $f\in \iF_{k,n}^c$. 
		
Finally, we prove that $\iF_{k,n}$ is dense. Let $\eps>0$ and a continuous $g\colon K\to \R^d$ be given. We will construct an $f\in B(g,2\eps)\cap \iF_{k,n}$. Choose $m\geq n$ with $k/m<\eps$ and finitely many distinct points $\{x_1,\dots,x_N\}$ in $K$ such that $\bigcup_{i=1}^N B(x_i,1/(2m))=K$. For all $1\leq i\leq N$ choose distinct points $z_i\in B(x_i, 1/(2m))\setminus \{x_1, \dots ,x_N\}$ such that  $|g(x_i)-g(z_i)|\leq \eps$. Define $f(x_i)=g(x_i)$ and $f(z_i)=y_i$ so that $|g(x_i)-y_i|=\eps$. Then $|f(z_i)-g(z_i)|\leq 2\eps$, so applying Tietze Extension Theorem for the finite set $\{x_i,z_i: i\leq N\}$ we obtain a continuous map $f\in B(g,2\eps)$ with the above property. We need to check that $f\in \iF_{k,n}$. Indeed, fix $x\in K$ arbitrarily. Then there exists $1\leq i\leq N$ such that $x_i,z_i\in B(x,1/m)$, and $|f(x_i) -f(z_i)|=\eps>k/m$. Thus $ \diam f(B(x,1/m))>k/m$, so $f\in \iF_{k,n}$. This completes the proof. 
	\end{proof}
	
Therefore the proof of Theorem~\ref{t:versions of singularity} is also complete. \end{proof}

\begin{theorem}
\label{t:length}
If $f \in \Homeo([0,1])$ then the five properties \eqref{f:derivative}-\eqref{f:full to null} are equivalent, and so is  
\begin{enumerate}  \setcounter{enumi}{5} \item \label{e:length} the length of the graph of $f$ equals $2$.
\end{enumerate}
\end{theorem}

\begin{proof}
By Theorem~\ref{t:versions of singularity} it is enough to prove that \eqref{f:full to null} $\Rightarrow$ \eqref{e:length} and \eqref{e:length} $\Rightarrow$ \eqref{f:derivative}. 
Fix a homeomorphism $f\colon [0,1]\to [0,1]$, we may assume that $f$ is monotone increasing. First we show that 
\begin{equation}\label{e:len} \length(\graph(f))\leq 2.
\end{equation} 
Indeed, consider finitely many points $0=a_0<a_1<\dots<a_k=1$. Then for the corresponding approximation of the length of $\graph(f)$ we obtain 
\begin{equation} \label{e:=2} \sum_{i=1}^k |(a_{i}-a_{i-1}, f(a_{i})-f(a_{i-1}))|\leq \sum_{i=1}^k (a_{i}-a_{i-1})+(f(a_{i})-f(a_{i-1}))=2,
\end{equation}
which implies \eqref{e:len}. 

Now we prove that \eqref{f:full to null} $\Rightarrow$ \eqref{e:length}: Let $f \in \Homeo([0, 1])$ be a homeomorphism satisfying \eqref{f:full to null}. Since Fact~\ref{f:wk} implies that   $\length(\graph(f))=\iH^1(\graph(f))$, by \eqref{e:len} it is enough to prove that 
\begin{equation} \label{e:H}  
\iH^1(\graph(f))\geq 2.
\end{equation} 
There is a Borel nullset $N\subseteq [0,1]$ with $\lambda(f(N))=1$. As the $\iH^1$ measure cannot increase under an orthogonal projection, projecting $\graph(f|_N)$ to the $y$-axis implies that 
\begin{equation*}
\iH^1(\graph(f|_N))\geq 1.
\end{equation*} 
Similarly, projecting $\graph(f|_{[0,1]\setminus N})$ to the $x$-axis yields 
\begin{equation*} \iH^1(\graph(f|_{[0,1]\setminus N}))\geq 1.
\end{equation*} 
Adding up the above two inequalities implies \eqref{e:H}. Here we used that the two parts of the graphs are continuous one-to-one images of Borel sets, so they are Borel (in particular, $\iH^1$ measurable) by \cite[Corollary~15.2]{Ke}.

Now we prove that \eqref{e:length} $\Rightarrow$ \eqref{f:derivative}: Let $f \in \Homeo([0, 1])$ be a homeomorphism satisfying \eqref{e:length}. Since the length of the graph of $f$ is $2$, for each $n\in \N^+$ we can choose points 
\begin{equation*}
0=a^n_0<a^n_1<\dots<a_{k_n}^n=1
\end{equation*} such that 
\begin{equation} \label{e:a_n} 
a_i^n-a_{i-1}^n< \frac 1n \text{ for all } 1\leq i\leq k_n
\end{equation} and the sum of the lengths of the corresponding line segments satisfies
\begin{equation} \label{e:ell}
\ell_n \defeq \sum_{i=1}^{k_n} \left|(a_i^n-a_{i-1}^n, f(a_i^n)-f(a_{i-1}^n))\right|\geq 2-2^{-n}.
\end{equation} 
Define 
\begin{equation*}
S_n=\bigcup \left\{(a_{i-1}^n,a_i^n): 1\leq i\leq k_n \text{ and } \frac{f(a_i^n)-f(a_{i-1}^n)}{a_i^n-a_{i-1}^n}\leq \frac 1n \right\}.
\end{equation*}
It is enough to prove that $\lambda(S_n)\to 1$ as $n\to \infty$. Indeed, let $S=\limsup_{n} S_n$, then clearly $\lambda(S)=1$. By the definition of $S$ and \eqref{e:a_n} at each point $x\in S$ the \emph{lower derivative}
\begin{equation*}
\underline{D}f(x)\defeq \liminf_{t\to 0} \frac{f(x+t)-f(x)}{t}
\end{equation*}
satisfies $\underline{D}f(x)=0$. Since $f$ is monotone, it is differentiable almost everywhere, so we obtain that $f'(x)=0$ at almost every $x$. 

Finally, we prove that $\lambda(S_n)\to 1$ as $n\to \infty$. For each $n\in \N^+$ define
\begin{equation*}
\iI_n=\left\{ 1\leq i\leq k_n:  \frac{f(a_i^n)-f(a_{i-1}^n)}{a_i^n-a_{i-1}^n}> \frac 1n \right\}.
\end{equation*}
Fix $n\in \N^+$, we clearly have
\begin{equation} \label{e:S}
\sum_{i\in \iI_n} a_i^n-a_{i-1}^n=1-\lambda(S_n). 
\end{equation}
For all $1\leq i\leq k_n$ define 
\begin{equation*}
w_{i}^n=(a_i^n-a_{i-1}^n)+(f(a_i^n)-f(a_{i-1}^n))-|(a_i^n-a_{i-1}^n,f(a_i^n)-f(a_{i-1}^n)|.
\end{equation*}
For every $i\in \iI_n$ we obtain by elementary calculation that
\begin{equation} \label{e:w}
 w_i^n \geq (a_i^n-a_{i-1}^n)\left(1+1/n-\sqrt{1+1/n^2}\right)\geq \frac{a_i^n-a_{i-1}^n}{n+1}. 
\end{equation}
Then \eqref{e:ell}, the equation in \eqref{e:=2}, \eqref{e:w}, and \eqref{e:S} imply that 
\begin{equation*}
2^{-n}\geq 2-\ell_n=\sum_{i=1}^{k_n} w_i^n\geq \sum_{i\in \iI_n} \frac{a_i^n-a_{i-1}^n}{n+1}=\frac{1-\lambda(S_n)}{n+1}. \end{equation*}
Thus $\lambda(S_n)\geq 1-(n+1)2^{-n}\to 1$ as $n\to \infty$, which completes the proof.
\end{proof}

\section{Singularity of generic homeomorphisms}
\label{s:gen_sing}

In this section we prove that for each $d\geq 2$ the generic $f\in \HomeoId$ is not strongly singular, but for all $d\geq 1$ it is singular. This answers the generic case of Questions \ref{q:sing} and \ref{q:strong_sing}. The next result implies that the generic $f\in  \HomeoId$ is not strongly singular.

\begin{theorem} \label{t:gen} Let $d\geq 2$ be an integer. The generic $f\in  \HomeoId$ is nowhere differentiable. 
\end{theorem}

We will only need the full power of the following theorem in the next section. Here we only use it for the one-element compact sets $\iK=\{g\}$. 

\begin{theorem} \label{t:nowhere} 
Let $d\geq 2$ be an integer. Let $\iK\subseteq \HomeoId$ be a compact set and let $\varepsilon>0$. Then there exists a homeomorphism $f\in B(\id, 2\varepsilon)$ such that $g\circ f$ is nowhere differentiable for all $g\in \iK$. Moreover, for all $g\in \iK$ and $x\in [0,1]^d$,
\begin{equation} \label{eq1}
\limsup_{y\to x} \frac{|g(f(x))-g(f(y))|}{|x-y|}=\infty.
\end{equation} 
\end{theorem}

First we prove the following elementary lemma. We will only use its consequence Corollary~\ref{c:easy} in this section, the more general statement will be needed in Section~\ref{s:Hd}.   

\begin{lemma} \label{l:s_n}
Assume that $s_n\downarrow 0$. Then there exist a sequence $q_n\downarrow 0$ and a continuous function $\varphi\colon [0,1]\to [0,1]$ such that for all $n\in \N$ if $I,J$ are intervals of length $2^{-n}$ and $s_n$, respectively, then 
\begin{equation*}
\lambda(\{z\in I: \varphi(z)\in J\})\leq q_n\lambda(I).
\end{equation*}
\end{lemma}

\begin{corollary}\label{c:easy} Assume that $s_n\downarrow 0$. Then there exist $N\in \N$ and a continuous function $\varphi \colon [0,1]\to [0,1]$ such that for all $n\geq N$ the oscillation of $\varphi$ on any interval of length $2^{-n}$ is at least $s_n$.
\end{corollary}

\begin{proof}[Proof of Lemma~\ref{l:s_n}]
We call a closed interval $I \subseteq [0, 1]$ \emph{$m$-dyadic} if $\ell = [\frac{\ell}{2^m}, \frac{\ell+1}{2^m}]$ for some $\ell \in \N$, and we call it \emph{dyadic} if it is $m$-dyadic for some $m \in \N$. Let us fix a strictly increasing sequence $a_m$ of natural numbers such that
\begin{equation}
\label{e:am choice}
s_{a_m} \le 2^{-2^{m+1}}
\end{equation} 
for all $m\in \N$.
We construct piecewise linear continuous functions $\varphi_m\colon [0,1]\to [0,1]$ such that $\varphi_0\equiv 0$ and for all $m \geq 1$ we have
\begin{enumerate} 
\item \label{i1} $\|\varphi_{m}-\varphi_{m-1}\| \le 2^{-2^{m-1}}$,
\item \label{i2} if $I$ is an $a_m$-dyadic interval, and $J$ is a $2^{m}$-dyadic interval, then 
\begin{equation*} \lambda\left(\{x \in I : \varphi_{m}(x) \in J\}\right) \leq 2^{-2^{m-1}} \lambda(I).
\end{equation*}
\end{enumerate}
First we show that this is indeed enough. By \eqref{i1} the functions $\varphi_m$ uniformly converge, let $\varphi= \lim_{m \to \infty} \varphi_m$ be their limit. Then clearly $\varphi \colon [0,1]\to [0,1]$ is continuous. To check that $\varphi$ satisfies the lemma with some $q_n$, let $I$ and $J$ be intervals of length $2^{-n}$ and $s_n$ for some $n$, respectively. Clearly, if $n < a_0$, then we can set $q_n = 1$ to satisfy the conclusion of the lemma. Otherwise, let $m$ be chosen so that $a_m \le n < a_{m+1}$. Then one can find a sequence of $a_{m+1}$-dyadic intervals $I_1, \dots, I_k$ such that $I \subseteq \bigcup_{i=1}^k I_i$ and $I' = \bigcup_{i =1}^k I_i$ satisfies
\begin{equation} \label{e:ii}
 \lambda(I')\leq 3\lambda(I).
\end{equation}  
A $s_n$ is decreasing, \eqref{e:am choice} yields 
\begin{equation*} s_n  \le s_{a_m} \le 2^{-2^{m+1}}. 
\end{equation*} 
Hence, we can find a sequence of consecutive $2^{m+1}$-dyadic intervals $J_1, J_2, \dots, J_6$, so that $J \subseteq J_3 \cup J_4$. Let $J' = \bigcup_{j=1}^6 J_j$. 
By \eqref{i1} we obtain $\|\varphi - \varphi_{m+1}\| \le 2^{1-2^{m+1}}$, hence 
\begin{equation} \label{e:subs} 
\{x\in I : \varphi(x) \in J\} \subseteq \{x\in I' : \varphi(x) \in J\} \subseteq \{x\in I' : \varphi_{m+1}(x) \in J'\}.
\end{equation}
Using \eqref{e:subs}, applying \eqref{i2} to each pair of intervals $I_i$ and $J_j$, and \eqref{e:ii} yield that
\begin{align*} 
\lambda\left(\left\{x \in I : \varphi(x) \in J\right\}\right) &\leq \lambda\left(\left\{x \in I' : \varphi_{m+1}(x) \in J'\right\}\right) \\
&\leq 6\cdot 2^{-2^{m}} \lambda(I') \\
&\leq 18 \cdot 2^{-2^{m}} \lambda(I).
\end{align*} 
We can then define $q_n = 18\cdot 2^{-2^m}$ for the largest $m$ with $a_m \le n$. One can easily check that $q_n \downarrow 0$, which will finish the proof of the lemma. 

Finally, we construct the functions $\varphi_m$ satisfying \eqref{i1} and \eqref{i2}. If $m=0$ then $\varphi_0 \equiv 0$ is continuous and linear, while \eqref{i1} and \eqref{i2} are vacuous. Now let $m \ge 1$, and suppose that $\varphi_0, \dots, \varphi_{m-1}$ are already constructed satisfying our conditions. Our goal is to construct $\varphi_m$. Using that $\varphi_{m-1}$ is piecewise affine, we can decompose $[0,1]$ into  non-overlapping closed intervals $B_1, \dots, B_{\ell}$ such that $\varphi_{m-1}(B_i) \subseteq J'_i$ for some $2^{m-1}$-dyadic interval $J'_i$ for each $1\leq i\leq \ell$. Clearly, $[0, 1]$ is covered by non-overlapping  intervals of the form $I \cap B_i$, where $I$ is an $a_m$-dyadic interval and $1\leq i \le \ell$. Then it suffices to define $\varphi_m$ on each such interval separately, with the condition that $\varphi_m$ and $\varphi_{m-1}$ coincide on the boundary of such intervals. 
	
So fix such intervals $I$ and $B_i$. Let $J'_i$ be a $2^{m-1}$-dyadic interval such that $\varphi_{m-1}(B_i) \subseteq J'_i$. Clearly, one can define $\varphi_m$ on $I \cap B_i$ so that 
\begin{enumerate}[(i)]
\item \label{i} $\varphi_m|_{\partial (I\cap B_i)}=\varphi_{m-1}|_{\partial (I\cap B_i)}$,
\item \label{ii} $\varphi_m$ is piecewise linear,
\item \label{iii} $\varphi_m(I \cap B_i) \subseteq J'_i$,
\item \label{iv} for each $2^m$-dyadic interval $J$ with $J \subseteq J'_i$ we have 
\begin{equation*} \lambda(\{x \in I \cap B_i : \varphi_m(x) \in J\}) =2^{-2^{m-1}} \lambda(I \cap B_i).
\end{equation*} 
\end{enumerate}
To see that \eqref{iv} can indeed be satisfied, note that there are exactly $2^{2^{m-1}}$ many $2^m$-dyadic intervals inside $J'_i$. By \eqref{i} and \eqref{ii} our function $\varphi_m$ is a well defined, piecewise linear, continuous function. Then $\varphi_m$ satisfies \eqref{i1} because of \eqref{iii}, and \eqref{i2} follows from \eqref{iv}. This completes the proof. 
\end{proof}

\begin{proof}[Proof of Theorem~\ref{t:nowhere}] Fix $\iK$ and $\eps$ as above. 
Let $B=B(\mathbf{0},1)\subseteq \R^2$ be the closed unit disc on the plane and consider $T =B\times [0,1]^{d-2}$. Observe that $[0,1]^d$ is bi-Lipschitz equivalent to $T$. Indeed, it is enough to give a bi-Lipschitz map between $[-1,1]^2$ and $B$, for which take the radial homeomorphism which maps $\partial [-r,r]^2$ onto $\partial B(\mathbf{0},r)$ for all $0\leq r\leq 1$. Therefore, it is enough to prove \eqref{eq1} for $T$. 
For all $n\in \N^+$ define 
\begin{equation*}
s_n=\min\{0\leq s\leq 1: |g(t_1)-g(t_2)|\geq 1/n \text{ whenever } g\in \iK \text{ and } |t_1-t_2|\geq s\}.   
\end{equation*}
By the compactness of $\iK$ we have $s_n>0$ for all $n$, and $s_n \downarrow 0$ as $n\to \infty$. By Corollary~\ref{c:easy} there exist a continuous function $\varphi\colon [0,1]\to [0,\eps]$ and $N\in \N$ such that for all $n\geq N$ the oscillation of $\varphi$ on any interval of length $2^{-n}$ is at least $4\sqrt{s_n}$. By increasing $N$ if necessary we can assume that $\max\{2^{-n},4\sqrt{s_n}\}<\eps$ for all $n\geq N$. Let $h\colon [0,1]\to [0,1]$ be an increasing homeomorphism such that $h(0)=0$, $h(r)=r$ if $r>\eps$, and $h(2^{-n})=s_n$ for all $n\geq N$. Let us parametrize  
\begin{equation*}
T=\{(r,\alpha,y): 0\leq r\leq 1,\, 0\leq \alpha <2\pi,\, y\in [0,1]^{d-2}\}.
\end{equation*} 
Now we can define $f\colon T\to T$ by
\begin{equation*}
f(r, \alpha, y)=(h(r), \alpha+\varphi(r) \mod 2\pi, y).   
\end{equation*}
It is clear that $f$ is a homeomorphism. First we show that $f\in  B(\id, 2\varepsilon)$. Indeed, 
\begin{equation*}
|f(r, \alpha, y)-(r,\alpha, y)|\leq |h(r)-r|+\max\{h(r),r\} \varphi(r)\leq |h(r)-r|+\varphi(r)\leq  2\eps,
\end{equation*}
since $|h(r)-r|\leq \eps$ and $\varphi(r)\leq \eps$ by the definitions.

Now we prove \eqref{eq1}. Fix $t=(r,\alpha, y)\in T$. It is enough to find for all $n\geq N$ a $t_n\in T$ such that $|t-t_n|\leq 2^{-n}$ and $|f(t)-f(t_n)|\geq s_n$, then $|g(f(t))-g(f(t_n))\geq 1/n$ by the definition of $s_n$, so \eqref{eq1} clearly holds. First assume that $r=0$. Let $t_n=(2^{-n},\alpha,y)$, then $|t-t_n|=2^{-n}$, and $|f(t)-f(t_n)|\geq |h(0)-h(2^{-n})|=s_n$ for all $n\geq N$, and we are done.  
Finally, assume that $r>0$. Let $t_n=(r_n,\alpha, y)$ such that $|r-r_n|\leq 2^{-n}$ and $|\varphi(r)-\varphi(r_n)|\geq 2\sqrt{s_n}$. By elementary geometry
\begin{equation*}
|f(t)-f(t_n)|\geq \min\{r,r_n\} 2\sin(\sqrt{s_n}) \geq r\sqrt{s_n}\geq s_n
\end{equation*}
for all large enough $n$ using that $s_n\to 0$ and $r_n\to r$ as $n\to \infty$. This completes the proof. 
\end{proof}

\begin{proof}[Proof of Theorem~\ref{t:gen}]
Define 
\begin{equation*}
\iG=\left\{f \in \HomeoId: \limsup_{y\to x} \frac{|f(x)-f(y)|}{|x-y|}=\infty \text{ for all } x\in [0,1]^d\right\}.
\end{equation*}
The elements of $\iG$ are clearly nowhere differentiable. Theorem~\ref{t:nowhere} yields that $\iG$ is dense in $\HomeoId$, so it is enough to prove that $\iG$ is $G_{\delta}$. It is easy to show that $\iG=\bigcap_{n=1}^{\infty} \iG_n$, where 
\begin{equation*}
\iG_n=\{f\in \HomeoId: \forall x\in [0,1]^d \, \exists y\in U(x,1/n) \text{ s.\,t. } |f(x)-f(y)|>n|x-y|\}. 
\end{equation*}
Fix $n\in \N^+$, it is sufficient to show that $\iG_n^c$ is closed. Let $f_k\in \iG_n^c$ be a sequence such that $f_k\to f$ uniformly as $k\to \infty$, we need to prove that $f\in \iG_n^c$. By definition, for each $k$ there exists $x_k\in [0,1]^d$ such that $|f_k(x_k)-f_k(y)|\leq n|x_k-y|$ for all $y\in U(x_k,1/n)$. By choosing a convergent subsequence we may assume that $x_k\to x$ as $k\to \infty$. Let $y\in U(x,1/n)$ be arbitrarily fixed. Then $y\in U(x_k, 1/n)$ for all large enough $k$, so 
\begin{equation*}
|f(x)-f(y)|=\lim_{k\to \infty} |f_k(x_k)-f_k(y)|\leq \lim_{k\to \infty} n|x_k-y|=n|x-y|.
\end{equation*}
Thus $f\in \iG_n^c$, and the proof is complete. 
\end{proof} 

Having dealt with strong singularity, now we turn to singularity.

\begin{theorem} \label{t:gensing1}
 For each $d\geq 1$ the generic $f\in \HomeoId$ is singular. 
\end{theorem}
 
Clearly, for each $d$ there exists a meager set $M\subseteq [0,1]^d$ with $\lambda^d(M)=1$. Thus Theorem~\ref{t:gensing1} follows from the statement below, which is probably known. In dimension $1$ it is indeed the case, see e.g.~\cite[Theorem 13.1]{Ox}. In any case, we include a proof for the sake of completeness.
 
 \begin{theorem} Assume that $d\geq 1$ and 
 $M\subseteq [0,1]^d$ is meager. Then the generic $f\in \HomeoId$ satisfies $\lambda^d(f(M))=0$.
 \label{t:gen_sing}
 \end{theorem} 
First we need some lemmas.
\begin{lemma}
For any ball $B(x, r) \subseteq [0,1]^d$ and $\varepsilon > 0$ there exists $h \in \HomeoId$ that fixes every point of the boundary of $[0,1]^d$ such that $\lambda^d( h(B(x, r)) ) > 1 - \varepsilon$.
\end{lemma}
    
\begin{proof}
By choosing $h$ to be a `radial' homeomorphism fixing $x$ and moving the points away from $x$ (but fixing the boundary) we can obtain that $h([0,1]^d \setminus B(x, r))$ is in an arbitrarily small neighborhood of the boundary of $[0,1]^d$, hence its measure can be arbitrarily small.
\end{proof}

\begin{lemma}
For every nowhere dense set $K \subseteq [0,1]^d$ and for every $\varepsilon, \delta > 0$ there exists $h \in \HomeoId$ in the $\delta$-neighborhood of the identity with $\lambda^d( h(K) ) < \varepsilon$.
\end{lemma}
    
\begin{proof}
Choose $n \in \N^+$ such that the cubes in the $\frac1n$-grid are of diameter less than $\delta > 0$. In every cube in this $\frac1n$-grid, since $K$ is nowhere dense, we can fix a ball that is disjoint from $K$. In every such cube apply the (scaled version of the) previous lemma to this ball. Then the obtained homeomorphisms together form a homeomorphism $h$ that clearly works.
\end{proof}

\begin{lemma}
For every nowhere dense compact set $C \subseteq [0,1]^d$ and for every $\varepsilon > 0$ the set $\{ f\in \HomeoId : \lambda^d(f(C)) < \eps\}$ is dense open.
\end{lemma}
    
\begin{proof}
This set is clearly open (by the outer regularity of Lebesgue measure), so it suffices to check that it is dense. Let $g \in \HomeoId$ and $\delta > 0$ be given. Apply the previous lemma with $K = g(C)$. Then $f = h \circ g$ is in the $\delta$-neighborhood of $g$, and $f(C) = h(g(C)) = h(K)$, hence $\lambda^d( f(C) ) = \lambda^d( h(K) ) < \varepsilon$.
\end{proof}

\begin{proof} [Proof of Theorem~\ref{t:gen_sing}]
Let $(C_n)_{n \in \N^+}$ be a sequence of nowhere dense compact sets in $[0,1]^d$ such that $M\subseteq \bigcup_{n=1}^{\infty} C_n$. Applying the previous lemma for every $C_n$ and for every $\varepsilon=\frac1k$ $(k \in \N^+)$ we obtain that $\lambda^d( f( M)) = 0$ holds for the generic $f \in \HomeoId$.
\end{proof}

\section{Singularity of prevalent homeomorphisms}
\label{s:prev_sing}
    
 In this section we prove that for all $d\geq 2$ the strongly singular homeomorphisms $f\in \HomeoId$ form a Haar ambivalent set and also that for every $d\geq 1$ the prevalent $f\in \HomeoId$ is singular. These results answer the prevalent case of Questions~\ref{q:sing} and \ref{q:strong_sing}.
 
\begin{theorem} \label{t:ambi}
Let $d\geq 2$ be an integer. The set 
\begin{equation*}
\iF=\{f\in    \HomeoId: f \text{ is strongly singular} \}
\end{equation*}
is Haar ambivalent. In fact, both $\iF$ and $\iF^c$ are compact catcher. 
\end{theorem}

Theorem~\ref{t:nowhere} implies the following. 

\begin{corollary} \label{c:dif} Let $d\geq 2$ be an integer. The set 
\begin{equation*}
\{f\in \HomeoId: f \text{ is nowhere differentiable}\}
\end{equation*}
is compact catcher, so Haar positive.
\end{corollary}

Theorem~\ref{t:ambi} follows from Corollary~\ref{c:dif} and Theorem~\ref{t:zero} below. First we need the following construction. 

\begin{definition} \label{d:SMC}
Let us define the \emph{Smith--Volterra--Cantor set} $K\subseteq [0,1]$ as follows. In the first step we remove the middle open interval of
length $\frac 14$ from $[0,1]$, and obtain two first level elementary intervals. After the $(n-1)$st step we have $2^{n-1}$ disjoint, closed $(n-1)$st level elementary intervals. In the $n$th step we remove the middle open intervals of length $2^{-2n}$ from each of them, and obtain $2^n$ disjoint, \emph{$n$th level elementary intervals}. We continue this procedure for all $n\in \mathbb{N}^{+}$, and the limit set is the Smith--Volterra--Cantor set $K$. A closed interval is an \emph{elementary interval of $K$} if it is an $n$th level elementary interval for some $n\geq 1$. We can define elementary intervals analogously for similar copies of $K$. Then 
\begin{equation*}
\lambda(K)=1-\sum_{n=1}^{\infty} 2^{n-1} \cdot 2^{-2n}=\frac 12.
\end{equation*}
For all $n\geq 1$ the length of the $n$th level elementary intervals equals 
\begin{equation} \label{e:bn}  
b_{n}=\frac{1}{2^n}\left(1-\sum
_{i=1}^{n}2^{i-1}\cdot 2^{-2i}\right)=2^{-(n+1)}+2^{-(2n+1)}\in [2^{-(n+1)},2^{-n}).
\end{equation}
\end{definition}

\begin{theorem} \label{t:zero}
The set
\begin{equation*}
\{ f\in    \HomeoId: f\text{ has zero derivative almost everywhere} \}
\end{equation*}
is compact catcher. 
\end{theorem}

\begin{proof} For the sake of simplicity endow $[0,1]^d$ with the maximum metric, and let $\iK\subseteq \HomeoId$ be an arbitrarily fixed compact set. We may assume that $\id\in \iK$. We will define $F\in \HomeoId$ such that $(g\circ F)'(x)=\mathbf{0}_{d\times d}$ for all $g\in \iK$ for almost every $x\in [0,1]^d$, where $\mathbf{0}_{d\times d}$ denotes the $d\times d$ zero matrix. Define $\varphi\colon [0,1] \to [0,1]$ as 
\begin{equation*}
\varphi(r)=\max\{0\leq s\leq 1: |g(x)-g(y)|\leq r^3 \text{ whenever } g\in \iK \text{ and } |x-y|\leq s\}.
\end{equation*}
Clearly, $\varphi$ is increasing, and by the compactness of $\iK$ we obtain that $\varphi(r)=0$ iff $r=0$. We will define a homeomorphism $f\colon [0,1]\to [0,1]$ and families of non-overlapping closed intervals $\{\iI_n\}_{n\geq 1}$ such that for each $n$ and $I\in \iI_n$ we have  
\begin{enumerate}
\item \label{e1}  $2^{-(n+1)}\leq \diam(I)<2^{-n}$,
\item \label{e2} $\diam f(I)\leq \varphi(\diam I)$, and
\item \label{e3} $\lambda(\liminf_{n} \bigcup \iI_n)=1$. 
\end{enumerate}
First we prove that this suffices. Let $G=\liminf_{n} \bigcup \iI_n$ and for each $x\in G$ and large enough $n$ choose $I_n(x)\in \iI_n$ such that $x\in I_n(x)$. Let 
\begin{equation*}
P=\{x\in G: B(x,4^{-n})\subseteq I_n(x) \text{ for all large enough } n\}.
\end{equation*}
First we show that $\lambda(P)=1$. Indeed, for every $n\geq 1$ define
\begin{equation*}
S_n=\bigcup \{U(\partial I_n, 4^{-n}): I_n\in \iI_n\}
\end{equation*}
and let $S=\limsup_n S_n$. Then $P=G\setminus S$. 
Since $\iI_n$ consists of at most $2^{n+1}$ intervals of $[0,1]$, we obtain that
\begin{equation*} \lambda(S_n)\leq 2^{n+2} 4^{-n}=2^{-n+2}.
\end{equation*}
Therefore $\lambda(S)=0$ by the Borel--Cantelli Lemma, hence $\lambda(P)=\lambda(G)=1$. Now we define $F\in \HomeoId$ by \begin{equation*}
F(x_1,\dots,x_d)=(f(x_1),\dots,f(x_d)).
\end{equation*}

We show that for all $g\in \iK$ we have $(g\circ F)'(x)=\mathbf{0}_{d\times d}$ at each $x\in P^d$. Fix $x\in P^d$ and $g\in \iK$ and, then $I_n(x_i)$ are defined such that $B(x_i,4^{-n})\subseteq I_n(x_i)$ for all large enough $n$ and $1\leq i\leq d$. Fix such an $n$ and let $B_n=B(x,4^{-n})=\prod_{i=1}^d B(x_i, 4^{-n})$. Then the definition of $F$, \eqref{e2}, and \eqref{e1} imply that $\diam F(B_n)\leq \varphi(2^{-n})$. Thus the definition of $\varphi$ yields that
\begin{equation*}
\diam (g\circ F)(B_n)\leq 2^{-3n},    
\end{equation*}
so indeed $(g\circ F)'(x)=\mathbf{0}_{d\times d}$.

Finally, we construct $\iI_n$ and $f$ satisfying \eqref{e1}, \eqref{e2}, and \eqref{e3}. Let $K\subseteq [0,1]$ be the Smith--Volterra--Cantor set, see Definition~\ref{d:SMC}. Let $f_0=\id$ and $K_0=\{0,1\}$. Assume by induction that an increasing homeomorphism $f_{n-1}\colon [0,1]\to [0,1]$ and a compact set $K_{n-1}\subseteq [0,1]$ have already been constructed for some $n\geq 1$ such that the length of each complementary interval of $K_{n-1}$ is an integer power of $2$. We will construct an increasing homeomorphism $f_n\colon [0,1]\to [0,1]$ and a compact set $K_n\subseteq [0,1]$ such that $K_{n-1}\subseteq K_n$ and $f_n|_{K_{n-1}}=f_{n-1}|_{K_{n-1}}$. We will obtain $K_n$ by filling up the complementary intervals of $K_{n-1}$ with similar copies of $K$. More precisely, enumerate the complementary intervals $\{(u_i,v_{i})\}_{i\geq 1}$ of $K_{n-1}$ and let us decompose each $[u_i,v_{i}]$ into finitely many non-overlapping closed intervals $\{J_{i,j}=[w_{i,j-1},w_{i,j}]\}_{1\leq j\leq k_i}$ such that $w_{i,0}=u_i$, $w_{i,k_i}=v_{i}$, and for all $1\leq j\leq k_i$ we have $w_{i,j}-w_{i,j-1}=2^{-N}$ for some large enough $N=N(i)\in \N^+$ for which the oscillation of $f_{n-1}$ satisfies 

\begin{equation} \label{e:osc}
f_{n-1}(w_{i,j})-f_{n-1}(w_{i,j-1})\leq 2^{-n}.  
\end{equation}
Let $K_{i,j}$ be the similar copy of $K$ with endpoints $\{w_{i,j-1},w_{i,j}\}$, more precisely, let $K_{i,j}=\psi_{i,j}(K)$ where $\psi_{i,j}(z)=(w_{i,j}-w_{i,j-1})z+w_{i,j-1}$. Define 
\begin{equation*}
K_n=K_{n-1}\cup \left( \bigcup_{i=1}^{\infty} \bigcup_{j=1}^{k_i} K_{i,j} \right).
\end{equation*}
Let us modify $f_{n-1}$ on each interval $J_{i,j}$ simultaneously as follows. Fix $i,j$, we will define $f_n|_{J_{i,j}}$ as a limit of uniformly convergent, strictly increasing functions $h_k=h_k(i,j)$. First let $h_0=f_{n-1}|_{J_{i,j}}$. If $h_{k-1}$ is given, we modify it on each $(k-1)$st level elementary interval $[a,b]$ of $K_{i,j}$ in a strictly monotone way such that $h_k(a)=h_{k-1}(a)$, $h_k(b)=h_{k-1}(b)$, and the $k$th elementary subintervals $I_1,I_2$ in $[a,b]$ satisfy 
\begin{equation} \label{e:hkJ} 
\diam h_k(I_m)\leq \varphi(\diam I_m)
\end{equation} for $m=1,2$. We can define $h_k$ for all $k\in \N$. Since the length of a $(k-1)$st level elementary interval is at most $2^{-(k-1)}$, for all $k\geq 1$ we obtain 
\begin{equation} \label{e:h}
\|h_{k}-h_{k-1}\|\leq \varphi (2^{-(k-1)})\leq 2^{-3(k-1)},
\end{equation} 
where the last inequality comes from $\id \in \iK$. Inequality \eqref{e:h} yields that $h_k$ uniformly converges to some continuous $h=h_{i,j}$. We show that $h$ is strictly increasing. Let $x,y\in J_{i,j}$ such that $x<y$, we need to prove that $h(x)<h(y)$. Choose a complementary interval $U$ of $K_{i,j}$ such that $U\subseteq (x,y)$. By our construction $h|_U=(h_k)|_U$ for a large enough $k$ and $h_k$ is strictly increasing, so $h(x)<h(y)$. By the construction for every $k\geq 1$ and $k$th level elementary interval $I=[c,d]$ of $K_{i,j}$ we have $h_{\ell}(c)=h(c)$ and $h_{\ell}(d)=h(d)$ for all $\ell\geq k$, so \eqref{e:hkJ} implies that 
\begin{equation} \label{e:hI}
 \diam h_{i,j}(I)\leq \varphi(\diam I).
\end{equation}
Define the homeomorphism $f_n \colon [0,1]\to [0,1]$ such that 
\begin{equation*}
f_n(z)=\begin{cases} 
f_{n-1}(z)  & \textrm{if } z\in K_{n-1},   \\
h_{i,j}(z) & \textrm{if } z\in I_{i,j} \textrm{ for some } i\geq 1 \text{ and } 1\leq j\leq k_i.
\end{cases}
\end{equation*}
Note that $f_n$ is indeed strictly increasing, as $f_{n-1}$ and $h_{i,j}$ are strictly increasing as well. We have defined $f_n$ for all $n\in \N$. The construction and \eqref{e:osc} imply
\begin{equation*}
\|f_{n}-f_{n-1}\|\leq 2^{-n}
\end{equation*}
for all $n\geq 1$, so $f_n$ converges to a homeomorphism $f\colon [0,1]\to [0,1]$. Indeed, it is easy to show that $f$ is strictly increasing: Let $0\leq x<y\leq 1$ be arbitrary, we need $f(x)<f(y)$. As $\bigcup_{n=1}^{\infty}K_n$ is dense in $[0,1]$, we can choose $n\in \N$ and $x_n,y_n\in K_n$ such that $x<x_n<y_n<y$. As $f|_{K_n}=(f_n)|_{K_n}$ by our construction and $f_n$ is strictly increasing, we obtain 
\begin{equation*} f(x)\leq f(x_n)=f_n(x_n)<f_n(y_n)=f(y_n)\leq f(y),
\end{equation*}
proving that $f$ is a homeomorphism. Let $\iI$ be the family of elementary intervals $I$ of all the copies of $K$ of the form $K_{i,j}$ that we used during the construction. We show that for all $I\in \iI$ we have 
\begin{equation} \label{e:iI} 
\diam f(I)\leq \varphi(\diam I).
\end{equation} 
Indeed, each $I\in \iI$ is an elementary interval of $K_{i,j}\subseteq K_{n}$ for some $i,j,n$. Then $\diam f_{m}(I)=\diam f_{n}(I)=\diam h_{i,j}(I)$ for all $m \geq n$ by the construction of $f_m$. Thus $\diam f(I)=\diam h_{i,j}(I)$, so \eqref{e:hI} implies \eqref{e:iI}. For each $n\geq 1$ let 
\begin{equation*} 
\iI_n=\{I\in \iI: 2^{-(n+1)}\leq \diam I< 2^{-n}\}.
\end{equation*} 
We need to show that $f$ and $\iI_n$ satisfy \eqref{e1}, \eqref{e2}, and \eqref{e3}. Property~\eqref{e1} follows from the definition of $\iI_n$, and \eqref{e:iI} yields \eqref{e2}. Hence it is enough to show \eqref{e3}. Let $G=\bigcup_{n=1}^{\infty} K_n$. The construction of $K_n$ implies that 
\begin{equation*}
\lambda(G)=\lim_{n\to \infty} \lambda(K_n)=\lim_{n\to \infty} 1-(1-\lambda(K))^n=
\lim_{n\to \infty} (1-2^{-n})=1.
\end{equation*}
We need to show that $G\subseteq \liminf_{n} \bigcup \iI_n$ (actually equality holds). Indeed, if $z\in G$, then $z$ is in a similar copy $C$ of $K$ with similarity ratio $2^{-m}$ for some $m \in \N$. For each $k\geq 1$ the $k$th elementary interval $I_k$ of $C$ containing $z$ satisfies $|I_k|=2^{-m} b_k$, so $2^{-(m+k)-1}\leq |I_k|<2^{-(m+k)}$ by \eqref{e:bn}. Thus $I_k\in \iI_{k+m}$, and $z\in \bigcup \iI_{k+m}$ for all $k\geq 1$. Thus $z\in \liminf_{n} \bigcup \iI_n$, and the proof is complete.
\end{proof}    
    
      Before proving the main result of this section we need some preparation. 
    \begin{lemma}\label{l:R}
    For every homeomorphism $g\in \HomeoId$ there exists a Borel set $R_g\subseteq[0,1]^d$ such that $\lambda^d(R_g)=1$ and if $N\subseteq R_g$ and $\lambda^d(N)=0$, then $\lambda^d(g(N))=0$ as well.
    \end{lemma}
    
    \begin{proof}
    Let $\mathcal{A}$ be a maximal disjoint family of Borel sets of positive measure with measure zero preimage. Then $R_g =[0,1]^d\setminus g^{-1}(\bigcup \mathcal{A})$ clearly works.
    \end{proof}
    
      We also need \cite[Theorem 17.25]{Ke}, which states the following.
    \begin{theorem} \label{t:Ke}
    Let $(X, \mathcal{S})$ be a measurable space, $Y$ a separable metrizable space, and $A \subseteq X\times Y$ a measurable set. Then the map
    \begin{equation*}
        \Phi : X\times \mathbf{P}(Y)\to [0, \infty),\qquad \Phi(x, \nu) = \nu(A_x)
    \end{equation*}
    is measurable for $\mathcal{S}\times \mathcal{B}(\mathbf{P}(Y))$, where $A_x=\{y\in Y: (x,y)\in A\}$ is the $x$-section of $A$, $\mathbf{P}(Y)$ is the set of probability measures on $Y$ endowed with the weak topology, and $\mathcal{B}(\cdot)$ stands for the Borel $\sigma$-algebra.
    \end{theorem}
    
Now we are ready to prove the following.
    
\begin{theorem}
\label{t:prevalent}
Let $d\geq 1$. The prevalent $f\in \HomeoId$ is singular. 
\end{theorem}
    
\begin{proof} Fix $d\geq 1$ and define
\begin{equation*}
\iF= \{f\in\HomeoId : f\text{ is singular}\},   
\end{equation*} 
we need to prove that $\iF$ is prevalent.
   
First we prove that $\iF$ is a Borel set. For each $n\in \mathbb{N}^+$ define 
\begin{equation*} 
g_n \colon \HomeoId\times[0,1]^d \to [0,\infty), \quad 
g_n(f, x) =\frac{\lambda^d\left(f\left(B\left(x,1/n\right)\right)\right)}{\lambda^d\left(B\left(x,1/n\right)\right)}. 
\end{equation*} 
By Theorem~\ref{t:versions of singularity} we obtain that
\begin{equation*}
\iF= \left\{ f\in\HomeoId : \limsup_{n\to \infty} g_n(f, x) = 0\text{ for a.\,e.}~x\in[0,1]^d \right\}.
\end{equation*}
    
Notice that the set
\begin{equation*}
\Gamma = \left\{(f, x) \in \HomeoId\times [0,1]^d: \limsup_{n\to \infty} g_n(f,x) = 0\right\}
\end{equation*}
is Borel. Indeed, it can be written as
\begin{equation*} \Gamma = \bigcap_{k=1}^\infty \bigcup_{m=1}^\infty \bigcap_{n=m}^\infty \left\{(f, x) \in \HomeoId\times [0,1]^d: g_n(f,x) <1/k\right\},
\end{equation*}
where the sets $\{(f,x): g_n(f,x)<1/k\}$ are open. Applying Theorem~\ref{t:Ke} for $X = \HomeoId$, $\mathcal{S}=\mathcal{B}(X)$, $Y = [0,1]^d$, and $A = \Gamma$ yields that
\begin{equation*} \Phi \colon \HomeoId\times \mathbf{P}([0,1]^d)\to [0, \infty),\quad \Phi(f, \nu) = \nu(\Gamma_f)
 \end{equation*}
is a Borel measurable map. The definition of $\Gamma$ implies that 
\begin{align*}
\iF&=\{f\in \HomeoId : \lambda^d(\Gamma_f)=1\}\\
&=\{f\in \HomeoId: (f,\lambda^d)\in \Phi^{-1}(\{1\})\}
\end{align*}
is the section of the Borel set $\Phi^{-1}(\{1\})$ at $\lambda^d\in \mathbf{P}([0,1]^d)$. Hence $\iF$ is a Borel set.
  
To conclude the proof, we will construct a measure $\mu$ which witnesses that $\iF$ is prevalent. Fix a singular homeomorphism $f_0\in \HomeoId$ and a subset $F\subseteq [0,1]^d$ such that $\lambda^d(F)=1$ and $\lambda^d(f_0(F))=0$. For $s= (s_1, s_2, \dots, s_d)$, where $s_i\in [1,2]$ for each $1\le i \le d$, let us define $\psi_{s} \in \HomeoId$ as
    \begin{equation*}
   \psi_{s}(x_1, x_2, \ldots, x_d) = (x_1^{s_1}, x_2^{s_2}, \ldots, x_d^{s_d}).
    \end{equation*}
    
    For any Borel set $B\subseteq \HomeoId$ define
     \begin{equation*} 
     \mu(B) = \lambda^{2d}\left(\left\{(s,t) \in [1,2]^d\times [1,2]^d: \psi_{s} \circ f_0 \circ \psi_{t} \in B\right\}\right).
    \end{equation*}
    It is easy to see that this defines a Borel probability measure on $\HomeoId$. We prove that $\mu$ is indeed a witness measure for $\iF$. Assume that $g, h\in \HomeoId$ are arbitrarily fixed, it is enough to show that $\mu(g^{-1}\circ \iF \circ h^{-1})=1$. 
    
    According to the definition of $\mu$,
    \begin{equation*} \mu(g^{-1}\circ \iF \circ h^{-1}) = \lambda^{2d}\left(\left\{(s,t) \in [1,2]^d\times [1,2]^d : \psi_{s} \circ f_0 \circ \psi_{t} \in g^{-1}\circ \iF \circ h^{-1}\right\}\right).
    \end{equation*}
    This means that we want to prove that for almost every pair $(s,t)$ the composition $g\circ \psi_{s} \circ f_0 \circ \psi_{t}\circ h$ is an element of $\iF$. Using the definition of $\iF$, this is equivalent to the condition that for almost every pair $(s,t)$ there exists a set $C_{s,t}\subseteq [0,1]^d$ such that $\lambda^d(C_{s,t})=1$ and 
    \begin{equation} \label{e:C}
\lambda^d((g\circ \psi_{s} \circ f_0 \circ \psi_{t}\circ h)(C_{s,t}))=0.        
\end{equation}
    
Suppose that $x\in (0,1)^d$. Since $h$ is a homeomorphism, $h(x)\in (0,1)^d$ as well. By definition $t\mapsto \psi_{t}(h(x))$ is a one-to-one, bi-Lipschitz map from $[1,2]^d$ into $(0,1)^d$. Thus $\lambda^d(F)=1$ yields
\begin{equation} \label{eq:psi1} \lambda^d\left(\left\{t\in [1,2]^d: \psi_{t}(h(x))\in F\right\}\right)=1.
\end{equation}
    
Now assume $x\in (0,1)^d$ and $t\in [1,2]^d$.    
By Lemma~\ref{l:R} there exists a Borel set $R_g\subseteq [0,1]^d$ such that $\lambda^d(R_g)=1$ and $\lambda^d(g(N))=0$ for every nullset $N\subseteq R_g$. As $f_0\circ \psi_{t}\circ h$ is still a homeomorphism, we have $(f_0\circ \psi_{t}\circ h)(x)\in (0,1)^d$. We may repeat the previous argument to show that 
\begin{equation} \label{eq:psi2} \lambda^d\left(\left\{s\in [1,2]^d: \psi_{s}((f_0\circ \psi_{t}\circ h)(x))\in R_g\right\}\right)=1.
\end{equation}

Applying \eqref{eq:psi1}, \eqref{eq:psi2}, and Fubini's theorem we obtain 
that almost every triple \begin{equation*} (x, s, t)\in [0,1]^d\times [1,2]^d \times [1,2]^d
\end{equation*}
satisfies that $\psi_{t}(h(x))\in F$ and $\psi_{s}((f_0\circ\psi_{t}\circ h)(x))\in R_g$.
    
Applying Fubini's theorem in the other direction yields that for almost every pair $(s,t)$ there exists a set $C_{s,t}\subseteq [0,1]^d$ with $\lambda^d(C_{s,t})=1$ satisfying
\begin{equation} \label{e:F} 
\psi_{t}(h(C_{s,t}))\subseteq F
\end{equation} 
and 
\begin{equation} \label{e:R}
(\psi_{s}\circ f_0\circ\psi_{t}\circ h)(C_{s,t})\subseteq R_g.
\end{equation}

We only need to show that $C_{s,t}$ satisfies \eqref{e:C}. By \eqref{e:F} we obtain
\begin{equation*} \lambda^d((f_0\circ\psi_{t}\circ h)(C_{s,t}))\le \lambda^d(f_0(F)) = 0.
\end{equation*}
Since $\psi_{s}$ is Lipschitz, it maps measure zero sets to measure zero sets, so
\begin{equation*}
\lambda^d((\psi_{s}\circ f_0\circ\psi_{t}\circ h)(C_{s,t}))= 0.
\end{equation*}
Then \eqref{e:R} and the definition of $R_g$ imply \eqref{e:C}, and the proof is complete. 
\end{proof}

\section{Solution to the problem of Mycielski and \texorpdfstring{$\iH^d$}{Hd}-measure of graphs in higher dimensions} \label{s:Hd}

In this section first we answer Mycielski's problem, then formulate the generalizations of this question and Banach's result to higher dimensions.

First, Theorems~\ref{t:prevalent} and \ref{t:length} immediately yield the following, answering Question~\ref{q:Mycielski} of Mycielski.

\begin{corollary}
\label{c:Mycielski}
The graph of the prevalent $f\in \Homeo([0, 1])$ is of length $2$.
\end{corollary}

New we turn to Question \ref{q:Hd}, the natural generalization to higher dimensions. First we answer this question in the generic case, then we partially answer it in the prevalent case. Somewhat surprisingly, the $\mathcal{H}^d$-measure tend to be infinite.

\begin{theorem} \label{t:inf}
For every $d \ge 2$ for the generic $f \in \HomeoId$ we have $\mathcal{H}^d(\graph(f)) = \infty$.
\end{theorem}

\begin{theorem}\label{t:ccinf} 
For every $d \ge 2$ the set 
\begin{equation*}
\iF=\{f\in \HomeoId: \mathcal{H}^d(\graph(f)) = \infty\}    
\end{equation*} 
is compact catcher, in particular not Haar null. 
\end{theorem}

We need to prove two lemmas first.

\begin{lemma}
\label{l:twist}
Let $d\geq 2$ and $C,\varepsilon \in \R^+$ be given. Then for every $f \in \HomeoId$ that is Lipschitz on a cube $Q \subseteq [0, 1]^d$ there exists $g \in \HomeoId$ that is also Lipschitz on $Q$ such that $g \in B(f,\eps)$ and $\mathcal{H}^d(\graph (g|_Q)) > C$.
\end{lemma}

\begin{proof}
By shrinking $Q$ if necessary, we can assume
that the distance between $Q$ and the boundary of $[0,1]^d$ is positive. Then, since $f$ is a homeomorphism, there exists $\delta \in \R^+$ such that
\begin{equation}
\label{e:delta}    
f(Q) \subseteq (\delta, 1-\delta)^d.
\end{equation}

First we check that $\lambda^d \left( \{ (x_1, \dots, x_d) \in Q : \det f'(x_1, \dots, x_d) \neq 0 \} \right) > 0$ (note that $f$ is Lipschitz on $Q$, and therefore differentiable almost everywhere on $Q$ by the Rademacher Theorem~\cite[Theorem~3.1.6]{Fe}). Indeed, if the determinant were $0$ almost everywhere on $Q$, then by the implication \eqref{f:derivative} $\Rightarrow$ \eqref{f:full to null} of Theorem \ref{t:versions of singularity} we could find a Borel set $F \subseteq Q$ with $\lambda^d(Q \setminus F)=0$ and $\lambda^d(f(F)) = 0$, but using that $f$ is Lipschitz on $Q$ we would also have $\lambda^d(f(Q \setminus F))=0$, hence $\lambda^d(f(Q))=0$, which is absurd since $f$ is a homeomorphism.

This implies that in particular $f_2'$ cannot be the zero vector almost everywhere on $Q$ (here $f_2$ is the second coordinate function of $f$). Hence we can find a set $P \subseteq Q$ with $\lambda^d(P) > 0$ such that $f_2'$ exists and $f_2' \neq \mathbf{0}$ on $P$. By partitioning $P$ into $d$ many pieces and picking the one with positive measure we can assume that there is a fixed coordinate $j$ such that $\frac{\partial f_2}{\partial x_j}$ exists and  $\frac{\partial f_2}{\partial x_j} \neq 0$ on $P$. And finally, by partitioning $P$ into countably many pieces and picking the one with positive measure we can assume that 
\begin{equation}
\label{e:P}
\begin{split}
\exists P \subseteq Q \text{ with } \lambda^d(P) > 0 \ \exists j \in \{1, \dots, d\} \ \exists \varrho \in \R^+ \text{ such that } \\ \frac{\partial f_2(x)}{\partial x_j} \text{ exists and } \left|\frac{\partial f_2(x)}{\partial x_j}\right| \geq \varrho \text{ for every } x\in P.
\end{split}
\end{equation}

The sets of the form $f_2^{-1} (a)$ (where $a$ ranges over $[0, 1]$) are pairwise disjoint and measurable, hence $\lambda^d(f_2^{-1} (a)) =0$ for all but countably many $a \in [0, 1]$.
Let $D = \{a \in [0, 1] : \lambda^d(f_2^{-1} (a)) =0 \}$. Then $D$ is clearly dense.

Let 
\begin{equation*} \Lip(f_1|_Q)=\min \{c: |f_1(x)-f_1(y)|\leq c|x-y| \text{ for all } x,y\in Q\}\in \R^+
\end{equation*} 
be the Lipschitz constant of the Lipschitz function $f_1|_Q$. As $D$ is dense, we can fix a continuous piecewise affine function $\varphi \colon [0,1] \to \R$ (a `zig-zag function') with the following properties:

\begin{enumerate}
\item 
\label{e:phi1}
$|\varphi| < \min\{ \eps, \delta \}$,
\item
\label{e:phi2}
the points of non-differentiability of $\varphi$ are in $D$,
\item
\label{e:phi3}
$|\varphi'| > \frac{1}{\varrho} \left(\frac{C}{\lambda^d(P)} +\Lip(f_1|_Q) \right)$ wherever the derivative exists,
\end{enumerate}
where $C$ is given in the statement of the theorem.

Next we slightly modify $P$. Since $f$ is differentiable almost everywhere, we can assume by removing a nullset from $P$ (i.e.~\eqref{e:P} will still hold) that 
\begin{equation}
\label{e:P2}
f \text{ is differentiable at every point of } P.
\end{equation}

Similarly, for the finitely many points $a$ of non-differentiability of $\varphi$ we remove the set $f_2^{-1} (a) \cap P$ from $P$. By \eqref{e:phi2} these sets are of measure zero, hence \eqref{e:P} still holds, but now we also have that 
\begin{equation}
\label{e:P3}
\varphi \text{ is differentiable at } f_2(x) \text{ for every } x \in P.
\end{equation}

Now we construct a piecewise affine homeomorphism $\Phi \in \HomeoId$ as follows. For $(y_1, ..., y_d) \in [0, 1]^d$ let
\begin{equation*} 
\Phi(y_1, ..., y_d) =  
\begin{cases} 
( y_1 (1 +\delta^{-1}\varphi(y_2)), y_2, \dots, y_d ) & \text{if } 0\leq y_1\leq \delta,\\
(y_1 + \varphi(y_2), y_2, \dots, y_d) & \text{if } \delta \leq y_1 \leq 1 - \delta,\\
(y_1+\delta^{-1} \varphi(y_2)(1-y_1), y_2, \dots, y_d) & \text{if } 1 - \delta \leq y_1 \leq 1.
\end{cases}
\end{equation*}

A short calculation shows that this map is well defined (the two values in the cases $y_1=\delta$ and $y_1 = 1 - \delta$ agree), and every segment of the form $[0, 1] \times \{(y_2, \dots, y_d)\}$ is an invariant set (note that $|\varphi| < \delta$ by \eqref{e:phi1}). Moreover, it is not hard to see that the map $\Phi$ is bijective
on every such segment, indeed, on $[\delta, 1 - \delta] \times \{(y_2, \dots, y_d)\}$ it is a translation by $\varphi(y_2)$, and on the remaining two small segments $[0, \delta] \times \{(y_2, \dots, y_d)\}$ and $[1 - \delta, 1] \times \{(y_2, \dots, y_d)\}$ it is the unique affine extension that makes it a continuous, (`strictly increasing') bijection of the segment $[0, 1] \times \{(y_2, \dots, y_d)\}$. Since these segments form a partition of $[0, 1]^d$, we obtain that $\Phi$ is a bijection of $[0,1]^d$. Next, the map $\Phi$ is easily seen to be continuous, since it is clearly continuous on the three rectangular boxes where it is defined separately, and since the values agree on the common faces where the rectangular boxes meet. And finally, by compactness, the inverse if $\Phi$ is also continuous, hence $\Phi \in \HomeoId$.

Another easy calculation shows, using that $|\varphi| < \eps$ by \eqref{e:phi1}, that $\Phi \in B(\id, \eps)$. Therefore, $\Phi \circ f \in B(f, \eps)$. Now define
\begin{equation*} 
g = \Phi \circ f.
\end{equation*}
Clearly, $g \in \HomeoId$, and (since $\Phi$ is piecewise affine, and hence Lipschitz), $g$ is Lipschitz on $Q$.
So all that remains to check is $\mathcal{H}^d(\graph (g|_Q)) > C$.

The Hausdorff measure of the graph of a function can be computed by the Area Formula \cite[Theorem~3.2.3]{Fe}. Let
\begin{equation*} 
G = \id \times g\colon [0, 1]^d \to [0, 1]^{2d},
\end{equation*}
that is,
\begin{equation*} 
G(x) =\\
(x, g(x)).
\end{equation*}

Then clearly $\graph (g|_Q) = \range (G|_Q)$. Since $G$ is Lipschitz and one-to-one on $Q$, we can apply the Area Formula which states that
\begin{equation*} 
\mathcal{H}^d(\range (G|_Q)) = \int_Q JG(x) \, \mathrm{d} \lambda^d(x),
\end{equation*}
where $JG(x)\defeq \sqrt{\det\left(G'(x)^T G'(x)\right)}$ is the Jacobian of $G$, which exists almost everywhere on $Q$. In order to show that 
\begin{equation*} 
\mathcal{H}^d(\graph (g|_Q)) = \mathcal{H}^d(\range (G|_Q)) = \int_Q JG(x) \, \mathrm{d} x > C
\end{equation*} 
it suffices to prove that
\begin{equation}
\label{e:JG}
JG(x) \text{ exists and } JG(x) > \frac{C}{\lambda^d(P)}
\end{equation}
at every point of the set $P \subseteq Q$ from \eqref{e:P} above.

As $G$ is Lipschitz on $Q$, by throwing away a final nullset from $P$ we can assume that $G'$, and hence also $JG$ exist at every point of $P$.

Let $x \in P$ be arbitrary. The 
Cauchy--Binet Formula \cite[page~9]{G} implies that $\det\left(G'(x)^T G'(x)\right)$ is the sum of the squares of the determinants of the $d \times d$ sized minors of $G'(x)$. Hence, in order to obtain \eqref{e:JG} it suffices to find a single minor of size $d \times d$ with the absolute value of its determinant greater than $\frac{C}{\lambda^d(P)}$. We claim that the minor obtained by taking the first $d$ rows (which form the $d \times d$ identity matrix) and replacing the $j$th row by the $(d+1)$st works (here $j$ comes from \eqref{e:P}). Some easy linear algebra shows that the determinant of this minor is the $j$th entry of its diagonal, that is, $\frac{\partial g_1(x_1, \dots, x_d)}{\partial x_j}$. Therefore, the proof will be complete if we show that at every point $x=(x_1,  \dots, x_d) \in P$ we have
\begin{equation}
\label{e:abs}
\left|\frac{\partial g_1(x)}{\partial x_j}\right| > \frac{C}{\lambda^d(P)}.
\end{equation}

By the definition of $\Phi$ and by \eqref{e:delta} for every $x=(x_1, ..., x_d) \in Q$ (and even on a neighborhood of $Q$) we have that 
\begin{equation*}
g(x) = (f_1(x) + \varphi(f_2(x)), f_2(x), ..., f_d(x)),
\end{equation*}
hence 
\begin{equation*}
g_1(x) = f_1(x) + \varphi(f_2(x)) \text{ on a neighborhood of } Q.
\end{equation*}
By \eqref{e:P2} and  \eqref{e:P3} these expressions are differentiable at every $x=(x_1,...,x_d) \in P$, hence by the Chain Rule
\begin{equation*}
\frac{\partial g_1(x)}{\partial x_j} = \frac{\partial f_1(x)}{\partial x_j} +
\varphi'(f_2(x)) \frac{\partial f_2(x)}{\partial x_j}.
\end{equation*}
Then \eqref{e:phi3}, \eqref{e:P}, and
$\left|  \frac{\partial f_1(x)}{\partial x_j} \right| \le \Lip(f_1|_Q)$
imply that
\begin{equation*}
\left| \frac{\partial g_1(x)}{\partial x_j}  \right|\geq 
\left| \varphi'(f_2(x)) \right|  \left| \frac{\partial f_2(x)}{\partial x_j} \right|
-\left|  \frac{\partial f_1(x)}{\partial x_j} \right|> \frac{C}{\lambda^d(P)}.
\end{equation*}
This completes the proof of \eqref{e:abs}, and hence the proof of the lemma.
\end{proof}

\begin{remark}
Note that the construction of the map $\Phi$ above was inspired by the so called `slides' from \cite{AK}.
\end{remark}

The following lower semicontinuity result seems to be known, the case of $d=2$ is due to Besicovitch \cite[page~21]{Be}. As we were not able to find a full reference (and for the reader's convenience), we include its proof here.  

\begin{lemma}
\label{l:lsc}
Let $d \ge 2$ and $C>0$. For every homeomorphism $f$ that is Lipschitz on a cube $Q \subseteq [0,1]^d$ satisfying $\mathcal{H}^d(\graph (f|_Q)) > C$ there exists $\varepsilon > 0$ such that for every $h\in B(f, \varepsilon)$ we have  $\mathcal{H}^d(\graph (h|_Q)) > C$.
\end{lemma}

Before proving Lemma~\ref{l:lsc} we need some preparation. 

\begin{definition}
Let $s>0$. A family $\mathcal{B}$ is said to be an \emph{$s$-almost Vitali covering} of a set $E\subseteq \R^m$ if for $\iH^s$ almost every $x\in E$ we have
\begin{equation*}
\inf\{ \diam B: x\in B,\, B\in \iB\}=0.
\end{equation*} 

Let $(\mathcal{K}(\R^m),d_{H})$ be the non-empty compact subsets of $\R^m$ endowed with the \emph{Hausdorff metric}, that is, for all compact sets $K_1,K_2\subseteq \R^m$ we have
\begin{equation*}
d_{H}(K_1,K_2)=\min \left\{r: K_1\subseteq B(K_2,r) \textrm{ and } K_2\subseteq B(K_1,r)\right\},    
\end{equation*}
where $B(A,r)=\{x\in \R^m: \exists y\in A \textrm{ such that } |x-y|\leq r\}$. Then $(\mathcal{K}(\R^m),d_{H})$ is a Polish space, see \cite{Ke} for more on this concept.

Let $K_n\subseteq \R^m$ be compact sets such that $K_n$ converges to $K$ in the Hausdorff metric. We say that $\{K_n\}_{n\geq 1}$ is \emph{almost uniformly concentrated in dimension $s$} if for every $\varepsilon > 0$ there is an $s$-almost Vitali covering $\mathcal{B}$ of $K$ such that for each $B\in \iB$, 
\begin{equation} \label{e:ucon}
\limsup_{n\to \infty} \iH^s (K_n\cap B) \geq (1 - \varepsilon)(\diam B)^s.
\end{equation}    
\end{definition}
    
For the proof of the following claim see \cite[Theorem~10.14]{Hausdorff-semicont} after the straightforward modifications.

\begin{claim} \label{cl}
Assume that $\{K_n\}_{n\geq 1}$ is an almost uniformly concentrated sequence of compact sets in $\R^m$ in dimension $s>0$ and $K_n$ converges to $K$ in the Hausdorff metric. Then 
\begin{equation*}
\liminf_{n\to \infty} \iH^s(K_n)\geq \iH^s(K).  \end{equation*}
\end{claim}

\begin{proof}[Proof of Lemma~\ref{l:lsc}]
Let $f\in \HomeoId$ be Lipschitz on a cube $Q \subseteq [0, 1]^d$ and assume that $\{f_n\}_{n\geq 1}$ is an arbitrary sequence of homeomorphisms such that $f_n\to f$ uniformly as $n\to \infty$. Then we need to prove that 
\begin{equation} \label{e:liminf} 
\liminf_{n\to \infty} \mathcal{H}^d(\graph (f_n|_Q))\geq \mathcal{H}^d(\graph (f|_Q)).
\end{equation} 

It is enough to show that $K_n\defeq \graph (f_n|_Q)$ is an almost uniformly concentrated sequence in dimension $d$. Indeed, we may assume that $Q$ is closed, so the compact sets $K_n$ converge to $K \defeq \graph (f|_Q)$ in the Hausdorff metric. Therefore, applying Claim~\ref{cl} for $K_n$, $K$, and $s=d$ will finish the proof of \eqref{e:liminf}. 

Let $\eps>0$ be arbitrarily fixed, we need to define a $d$-almost Vitali covering $\iB$ of $K$ for which \eqref{e:ucon} holds with $s=d$. Let 
\begin{equation*} D=\{x\in \inter Q: f \text{ is differentiable at } x\}.
\end{equation*}
For all $z\in D$ define $T_{z}\colon \R^d\to \R^d$ as 
\begin{equation*}
T_{z}(x)=f(z)+f'(z)(x-z).
\end{equation*}
Let us choose $0<\delta<\frac 12$ such that $(1-2\delta)^d\geq 1-\eps$. For all $z\in D$ define $\delta(z)>0$ such that $B(z,\delta(z))\subseteq [0,1]^d$ and for all  $x\in B(z,\delta(z))$ we have 
\begin{equation} \label{e:Tz}
|f(x)-T_z(x)|\leq \frac{\delta}{2} |x-z|.
\end{equation}
Define 
\begin{equation*}
\iB=\{B((z,f(z)),r): z\in D,\, 0<r<\delta(z)\}.    
\end{equation*}
By the Rademacher Theorem $\lambda^d(Q\setminus D)=0$. As $f|_Q$ is Lipschitz, we have \begin{equation*}
\iH^d(\{(x,f(x)): x\in Q\setminus D\})=0,    
\end{equation*} 
thus $\iB$ is really a $d$-almost Vitali covering of $K$.
  
Fix $z\in D$ and $B=B((z,f(z)),r)\in \iB$ with some $0<r<\delta(z)$. Let $T=T_z$ and $V=\{(x,T(x)): x\in \R^d\}$ be the tangent plane at $z$. Choose $N\in \N^+$ such that for all $n\geq N$ and $x\in [0,1]^d$ we have
\begin{equation} \label{e:fn}
    |f_n(x)-f(x)|\leq \frac{\delta}{2}r.
\end{equation}
It is enough to prove that 
\begin{equation} \label{e:N}
\iH^d(K_n\cap B)\geq (1-\eps)(\diam B)^d
\end{equation}
for all $n\geq N$. Fix an arbitrary $n\geq N$. By \eqref{e:fn} and \eqref{e:Tz} for all $x\in B(z,r)$ we obtain
\begin{equation} \label{e:dr}
|f_n(x)-T(x)|\leq |f_n(x)-f(x)|+|f(x)-T(x)|\leq \frac{\delta}{2}r + \frac{\delta}{2}r = \delta r. \end{equation}
Let $B'=V\cap B((z,f(z)),r(1-\delta))$ and $B''=V\cap B((z,f(z)),r(1-2\delta))$. Note that if $y\in B'$ then there is a unique $x\in B(z,r)$ such that $y=(x,T(x))$. Consider the continuous map $S\colon B'\to V$ such that for $y=(x,T(x))\in B'$ we have
\begin{equation} \label{e:pr}
S(y)=\pr_V ((x,f_n(x))),
\end{equation}
where $\pr_V$ is the orthogonal projection to $V$. By \eqref{e:dr} if $(x,T(x))\in B'$ then $(x,f_n(x))\in K_n\cap B$. Therefore, since projections cannot increase the Hausdorff measure, we obtain 
\begin{equation} \label{e:SB}
\iH^d(K_n\cap B)\geq \iH^d (S(B')).
\end{equation}
Let $y=(x,T(x))\in B'$, then $x\in B(z,r)$, \eqref{e:dr}, and \eqref{e:pr} imply
\begin{equation} \label{e:Sy}
|S(y)-y|\leq |(x,f_n(x))-(x,T(x))|=|f_n(x)-T(x)|\leq \delta r.    
\end{equation}
By \eqref{e:Sy} we can apply Fact~\ref{f:1} with $\alpha=\delta r$ and $\beta=r(1-\delta)$ for $S$ composed with a translation. This yields $B''\subseteq S(B')$. Thus
\begin{equation} \label{e:SB'}
 \iH^d (S(B'))\geq \iH^d (B'')=(\diam B'')^d
=((1-2\delta)2r)^d \geq (1-\eps)(\diam B)^d.
\end{equation}
Inequalities \eqref{e:SB} and \eqref{e:SB'} imply \eqref{e:N}, and the proof is complete.
\end{proof}

\begin{definition} Let us say that a map $f$ is \emph{somewhere smooth}, if it is smooth on a non-empty open set (which may depend on $f$).
\end{definition}

\begin{proof}[Proof of Theorem~\ref{t:inf}]
By Corollary \ref{c:approx} below, for every positive integer $d$ the set of somewhere smooth homeomorphisms is dense in $\HomeoId$. Note that a somewhere smooth map is Lipschitz on a suitable cube. Hence 
\begin{equation*}
\iF_n=\{f \in \HomeoId : \mathcal{H}^d(\graph(f)) > n\}    
\end{equation*} is dense for every $n \in \mathbb{N}$ by Lemma \ref{l:twist}, consequently $\iF_n$ contains a dense open set for every $n \in \mathbb{N}$ by Lemma \ref{l:lsc}, therefore 
\begin{equation*}
\{f \in \HomeoId : \mathcal{H}^d(\graph(f))=\infty \}=\bigcap_{n=0}^{\infty} \iF_n
\end{equation*} is co-meager.
\end{proof}

Finally, we prove Theorem~\ref{t:ccinf}. 

\begin{proof}[Proof of Theorem~\ref{t:ccinf}] Let $d\geq 2$ and a compact set $\iK\subseteq \HomeoId$ be fixed. We will construct a homeomorphism $\Phi\in \HomeoId$ such that $f\circ \Phi\in \iF$ for all $f\in \iK$. For all $n\in \N$ define
\begin{equation*}
s_n=\min\{0\leq s\leq 1: |f(x)-f(y)|\geq d2^{-n} \text{ whenever } f\in \iK \text{ and } |x-y|\geq s\}.
\end{equation*}
As $\iK$ is compact, it is easy to see that $s_n\downarrow 0$. By Lemma~\ref{l:s_n} there exist a sequence $q_n\downarrow 0$ and a continuous function $\varphi\colon [0,1]\to [0, \frac 14]$ such that for all $n\in \N$ if $I,J$ are intervals of length $2^{-n}$ and $s_n$, respectively, then 
\begin{equation} \label{e:z}
\lambda(\{z\in I: \varphi(z)\in J\})\leq q_n2^{-n}.
\end{equation}
Let $T=[\frac 14,\frac 12]\times [0,1]^{d-1}$. Similarly as in the proof of Lemma~\ref{l:twist} we can define a homeomorphism $\Phi\in \HomeoId$ such that for all $(x_1,\dots,x_d)\in T$ we have 
\begin{equation*}
\Phi(x_1,\dots,x_d)=(x_1+\varphi(x_2),x_2,\dots,x_d). 
\end{equation*}
Fix an arbitrary $f\in \iK$, we will show that 
\begin{equation} \label{e:fP}
\iH^d(\graph((f\circ \Phi)|_{T}))=\infty.
\end{equation}
Let $\mu$ be the pushforward measure $\lambda^d\circ \Psi^{-1}$, where $\Psi\colon T\to T\times [0,1]^d$ is defined as $\Psi(x)=(x,f(\Phi(x)))$. Then $\mu$ is a Borel measure supported on $\graph((f\circ \Phi)|_{T})$, and for any Borel set $A\subseteq [0,1]^{2d}$ we have 
\begin{equation*}
\mu(A)=\lambda^d(\{x\in T: (x,f(\Phi(x)))\in A\}).
\end{equation*}
Fix an arbitrary $n\in \N$ and cubes $Q_1\subseteq T$ and $Q_2\subseteq [0,1]^d$ of edge length $2^{-n}$, it is enough to prove that 
\begin{equation} \label{eq:mu}
\mu(Q_1\times Q_2)\leq q_n 2^{-nd}.
\end{equation}
Indeed, then the Mass Distribution Principle \cite[Theorem~4.19]{MP} implies that with some $c_d>0$ we have
\begin{equation*} 
\iH^d(\graph((f\circ \Phi)|_{T}))\geq c_d \liminf_{n\to \infty} \frac{1}{q_n} \lambda^{d}(T)=\infty,
\end{equation*}
so \eqref{e:fP} holds. In order to prove \eqref{eq:mu}, define the interval $J$ as the orthogonal projection of $f^{-1}(Q_2)$ to the first coordinate axis. Since $\diam Q_2<d2^{-n}$, the definition of $s_n$ implies that $\diam f^{-1}(Q_2)\leq s_n$, so $\diam J\leq s_n$ as well. Let $Q_1=I_1\times \cdots \times I_d$, by \eqref{e:z} for all $x_1\in I_1$ we obtain
\begin{equation} \label{e:x2}
\lambda(\{x_2\in I_2: \varphi(x_2)\in J-x_1\})\leq q_n 2^{-n}.
\end{equation}
It is straightforward that 
\begin{equation} \label{e:sub}
\{x\in Q_1: f(\Phi(x))\in Q_2\}\subseteq \{x\in Q_1: \varphi(x_2)\in J-x_1\}.
\end{equation}
Hence \eqref{e:sub} and \eqref{e:x2} easily yield
\begin{align*}
\mu(Q_1\times Q_2)&=\lambda^d(\{x\in Q_1: f(\Phi(x))\in Q_2\}) \\
&\leq \lambda^d(\{x\in Q_1: \varphi(x_2)\in J-x_1\}) \\
&=2^{-n(d-2)}\int_{I_1} \lambda(\{x_2\in I_2: \varphi(x_2)\in J-x_1\}) \,\mathrm{d} x_1 \\ 
&\leq q_n 2^{-nd}.
\end{align*}
Thus \eqref{eq:mu} holds, and the proof is complete. 
\end{proof}

\section{Concluding remarks and open problems}

In Theorem~\ref{t:inf} we showed that for $d\geq 2$ the $d$-dimensional Hausdorff measure of the graph of the generic $f \in \Homeo([0, 1]^d)$ is infinite. The next question asks weather this can be slightly improved.

\begin{question}
Let $d\geq 2$ be an integer. Is it true that $\iH^d|_{\graph(f)}$ is non-$\sigma$-finite for the generic $f \in \Homeo([0,1]^d)$?
\end{question}

Since $\graph(f) \subseteq [0,1]^{2d}$, it may even be possible that $\dim_H \graph(f)>d$ for the generic $f \in \Homeo([0, 1]^d)$, where $\dim_H$ is the Hausdorff dimension, see e.g.~\cite{Ma} for the definition. In fact, we can show that if $d\neq 4$ and $d\neq 5$ then the generic $f\in \HomeoId$ satisfies $\dim_H \graph(f)=d$, we plan to publish this in a forthcoming paper. As $\{f \in \Homeo([0, 1]^d) : \dim_H \graph(f)=d\}$ is $G_{\delta}$, the real question is the following.

\begin{question}
Let $d=4$ or $d=5$. Is it true that 
\begin{equation*} 
\{f\in \HomeoId: \dim_{H} \graph(f)=d\}
\end{equation*} is dense in $ \HomeoId$? 
\end{question}

For the case of prevalent homeomorphisms, we do not know whether the graph has infinite $d$-dimensional Hausdorff measure. 

\begin{question}
Let $d\ge 2$ be an integer. What is $\mathcal{H}^d(\graph(f))$ for the prevalent $f \in \Homeo([0,1]^d)$? What is $\dim_{H} \graph(f)$ for the prevalent $f \in \Homeo([0, 1]^d)$?
\end{question}

The answer to the following question may easily be known, but we have been unable to find it in the literature. By setting $Q = [0, 1]^d$ in Lemma \ref{l:lsc} one obtains a form of lower semi-continuity result for the $d$-dimensional Hausdorff measure of a graph of a homeomorphism $f \in \Homeo([0, 1]^d)$: if $f$ is Lipschitz with $\iH^d(\graph(f)) > c$ then there exists $\varepsilon > 0$ such that $\iH^d(\graph(g)) > c$ for all $g\in B(f,\eps)$. This observation motivates the following question. Since the case $d = 1$ is straightforward, we assume $d \ge 2$.

\begin{question}
Let $d \ge 2$ be an integer.
Is the $d$-dimensional Hausdorff measure lower semi-continuous restricted to the graphs of homeomorphisms of $[0, 1]^d$, that is, is the set \begin{equation*}
 \{f\in \HomeoId: \mathcal{H}^d(\graph(f)) > c\}
\end{equation*} open for each $c \in \R$?
\end{question}

\begin{remark}
Instead of using Hausdorff measures in the above problems, one may consider the notion of surface area introduced by Lebesgue \cite{L}. The \emph{Lebesgue area} $L(f)$ of a continuous map $f\colon [0,1]^d\to \R^d$ is obtained by taking the infimum of limit inferiors of surface areas of piecewise affine functions converging uniformly to $f$. Then $L$ is lower semi-continuous, and $L(f)=\iH^d(\graph(f))$ whenever $f$ is Lipschitz, see e.g.~\cite{Fe0} for more on this concept. Using our above results concerning the $\iH^d$ measure of the generic graph, it is straightforward to obtain that $L(f)=\infty$ holds for the generic $f\in \HomeoId$.  
\end{remark}

\section{Appendix (an approximation result in differential topology)}

We establish for convenience a known result needed above. Recall that a map $f$ is called somewhere smooth if it is smooth on a non-empty open set (that may depend on $f$).

\begin{theorem}
\label{t:approx}
Let $d \ge 2$ be an integer, $f \in \Homeo((0,1)^d)$ and $\varepsilon \colon (0,1)^d \to (0, \infty)$ be continuous. Then there exists  a somewhere smooth $g \in \Homeo((0,1)^d)$ such that $|f(x) - g(x)| < \varepsilon(x)$ for every $x \in (0,1)^d$.
\end{theorem}

In fact, an easy and elementary argument shows that the statement also holds for $d = 1$, but we will not need this here.

\begin{proof}
First we show that, if $d\ne4$, then the above $g$ can
even be a PL (piecewise affine) homeomorphism
of $(0,1)^d$. 

For $d=2$ see e.g.~\cite[Theorem 8.4]{Moise}, while for $d=3$ see \cite[Theorem 36.1]{Moise}.
For $d \ge 5$, let us fix a PL homeomorphism $\phi \colon  \R^d \to (0, 1)^d$ with 
\begin{equation}
    \label{e:phi is contraction}
    |\phi(x) - \phi(y)| \le |x - y| \text{ for each $x, y \in \R^d$}.
\end{equation}
Let $\widehat{\eps} = \varepsilon \circ \phi$. Then $\widehat{f} = \phi^{-1} \circ f \circ \phi$ is a homeomorphism of $\R^d$. We claim that $\widehat{f}$ can be $\widehat{\eps}$-approximated by a PL homeomorphism $h$, that is, 
\begin{equation}
    \label{e:h approximates f'}
    \left|\widehat{f} (x) - h(x)\right| < \widehat{\eps} (x) \text{ for each $x \in \R^d$}
\end{equation}
for some PL homeomorphism $h \in \Homeo(\R^d)$. It follows from a theorem of Connell and Bing (see e.g.~\cite[Theorem~4.11.1]{Rushing}) that every stable homeomorphism of $\R^d$ can be $\widehat{\eps}$-approximated by a PL homeomorphism, and since every orientation-preserving homeomorphism is stable if $d \ge 5$ by a result of Kirby \cite{Kirby}, we obtain an $\widehat{\eps}$-approximation $\widehat{h}$ of either $\widehat{f} $ or $r \circ \widehat{f}$, where $r$ is a reflection on a hyperplane. Let $h = \widehat{h}$ in the former case, while $h = r \circ \widehat{h}$ in the latter to finish the proof of the claim. Now set $g = \phi \circ h \circ \phi^{-1}$; clearly $g$ is a PL homeomorphism. Using \eqref{e:phi is contraction}, \eqref{e:h approximates f'} and the fact that $f = \phi\circ \widehat{f}  \circ \phi^{-1}$, it is easy to check that $g$ is also an $\varepsilon$-approximation of $f$.

Suppose now that $d=4$. Let $x \in (0, 1)^d$ be arbitrary, and let $V$ be an open neighborhood of $x$ small enough so that 
\begin{equation}
    \label{e:diam of f(V) is small}
    \diam f(V) < \inf\{\varepsilon(y) : y \in V\}.
\end{equation}
By \cite[Theorem 8.1A]{FQ} applied with $M = N = (0,1)^d$ (hence there is no boundary, so the restriction of smoothness on the boundary is vacuous, and the isotopies relative to the boundary are just isotopies), $K =\{x\}$ and $U = V$, we obtain that there exists a somewhere smooth $g \in \Homeo((0,1)^d)$ with $g|_{(0,1)^d\setminus V} = f|_{(0,1)^d\setminus V}$. Therefore $g(V) = f(V)$, and hence $g$ is an $\varepsilon$-approximation of $f$ by \eqref{e:diam of f(V) is small}. \end{proof}

\begin{corollary}
\label{c:approx}
Let $d$ be a positive integer, $f \in \HomeoId$ and $\varepsilon > 0$. Then there exists  a somewhere smooth $g \in \HomeoId$ 
with $g \in B(f, \varepsilon)$ such that $g$ agrees with $f$ on the boundary of $[0,1]^d$.
\end{corollary}

\begin{proof}
Choose a continuous $\eps \colon (0,1)^d \to (0, \infty)$ with $\varepsilon(x) < \varepsilon$ for every $x \in (0,1)^d$ such that $\varepsilon(x_n) \to 0$ whenever $x_n \to x \in \partial [0,1]^d$. Then obtain a map $g$ by applying Theorem \ref{t:approx} to this function $\varepsilon(\cdot)$ and $f|_{(0, 1)^d}$. Finally, extend this map $g$ to the boundary of $[0,1]^d$ so that $f$ and $g$ agree on the boundary.
\end{proof}

\subsection*{Acknowledgment}
We are indebted to J.~Luukkainen for providing the Appendix. He has actually also shown us that the set of homeomorphisms which are PL on a neighborhood of $\partial [0,1]^4$ is dense in $\Homeo([0,1]^4)$. We are grateful to D.~Nagy for some valuable remarks and helpful suggestions.

\end{document}